\documentclass[reqno, 12pt]{amsart}

\usepackage{tikz-cd}

\usepackage{amsfonts, amsthm, amsmath, amssymb}
\usepackage{hyperref}
\hypersetup{colorlinks=false}
\usepackage[all]{xy}

\usepackage[margin=1in]{geometry}

\usepackage{helvet}

\RequirePackage{mathrsfs} \let\mathcal\mathscr
  
\usepackage{mathtools}
\mathtoolsset{showonlyrefs}

\usepackage{hyperref}
\usepackage{dsfont}
\usepackage{upgreek}
\usepackage{mathabx,yfonts}
\usepackage{enumerate}

\numberwithin{equation}{section}

\newtheorem{theorem}{Theorem}[section] 
\newtheorem{lemma}[theorem]{Lemma}
\newtheorem{proposition}[theorem]{Proposition}

\newtheorem{conjecture}[theorem]{Conjecture}

\newtheorem{heuristic}[theorem]{Heuristic assumption}
\newtheorem{prediction}[theorem]{Prediction}

\theoremstyle{definition}
\newtheorem*{acknowledgements}{Acknowledgements}

\newtheorem{definition}[theorem]{Definition}

\newtheorem*{notation}{Notation}

\newtheorem{nnotation}[theorem]{Notation}

\renewcommand{\phi}{\varphi}

\newcommand{\0}{\mathbf{0}}

\newcommand{\bk}{|\b{k}|_1}
\newcommand{\bu}{\mathbf{u}}
\newcommand{\bv}{\mathbf{v}}

\newcommand{\mcl}{\mu_{\text{CL}}}
\newcommand{\mseq}{\mu_{\mathrm{seq}}}

\renewcommand{\leq}{\leqslant}

\renewcommand{\geq}{\geqslant}

\renewcommand{\c}{\mathbf{c}}

\renewcommand{\b}{\mathbf{b}}

\renewcommand{\r}{\mathbf{r}}

\renewcommand{\Im}{\mathrm{Im}}

\DeclareMathOperator{\Id}{Id} 

\DeclareMathOperator{\rk}{rk} 
\DeclareMathOperator{\isom}{Isom}
\DeclareMathOperator{\cl}{Cl}
\DeclareMathOperator{\disc}{D}

\DeclareMathOperator{\rank}{rk} 
\DeclareMathOperator{\Hom}{Hom} 
\DeclareMathOperator{\ext}{Ext} 
\DeclareMathOperator{\epi}{Epi} 
\DeclareMathOperator{\aut}{Aut}

\DeclareMathOperator{\Gal}{Gal}

\DeclareMathOperator{\n}{N}

\let\emptyset\varnothing

\DeclareSymbolFont{bbold}{U}{bbold}{m}{n}
\DeclareSymbolFontAlphabet{\mathbbold}{bbold}

\newcommand{\md}[1]{  \left(\textnormal{mod}\ #1\right)}
\newcommand{\lab}{\label} 

\newcommand{\Q}{\mathbb{Q}}
\newcommand{\bw}{\mathbf{w}}
\newcommand{\F}{\mathbb{F}}
\newcommand{\N}{\mathbb{N}}
\newcommand{\R}{\mathbb{R}}

\newcommand{\Z}{\mathbb{Z}}
\renewcommand{\l}{\left}
\newcommand{\m}{\mathfrak{m}}
\renewcommand{\r}{\right}
\renewcommand{\b}{\mathbf}
\renewcommand{\c}{\mathcal}
\renewcommand{\epsilon}{\varepsilon}

\renewcommand{\leq}{\leqslant}
\renewcommand{\geq}{\geqslant}

\DeclareMathOperator*{\Osum}{\sum{}^*}

\title[$4$-ranks and the general model of ray class groups] 
{$4$-ranks and the general model for statistics of ray class groups of imaginary quadratic number fields} 

\author{C. Pagano}
\address{Mathematisch Instituut\\ 
Universiteit Leiden \\ 
Leiden \\ 
2333 CA \\ 
Netherlands} 
\email{c.pagano@math.leidenuniv.nl}

\author{E.Sofos}
\address{
Max Planck Institute for Mathematics\\
Vivatsgasse 7, Bonn, 53111, Germany}
\email{sofos@mpim-bonn.mpg.de}

\subjclass[2010]{11R65, 11R29, 11R11, 11R45}
\date{\today}

\newcommand{\beq}[2]
{
\begin{equation}
\label{#1}
{#2}
\end{equation}
}

\begin{document}

\begin{abstract}
We extend the Cohen--Lenstra heuristics to the 
setting
of ray class groups of imaginary quadratic number fields, viewed as exact sequences of Galois modules. By asymptotically estimating 
the mixed moments governing the distribution  of a cohomology map,
we prove
these conjectures in the case of $4$-ranks. 
\end{abstract}

\maketitle

\setcounter{tocdepth}{1}
\tableofcontents
 
\section{Introduction}

Let $c$ be a positive odd square-free integer. 
Partition the set of its prime divisors, $S$, into $S_1 \cup S_3$, where if $l \in S_i$ then $l \equiv  i \md{4}$. For an imaginary quadratic number field $K$, denote by $\cl(K,c)$ the ray class group of $K$ of conductor $c$, and by $\disc(K)$ the discriminant of $K$. 
Let $j_1$ and $j_2$ be two non-negative integers. The following theorem will be shown to be a special case of the present work.
\begin{theorem}\label{thm 4rk:special case}
Consider all
imaginary quadratic number fields $K$
such that $\disc(K) \equiv  1  \md{4}$
and
$\mathcal{O}_K/c \cong_{\emph{\text{ring}}} \prod_{l \in S} \mathbb{F}_{l^2}$. 
When such $K$ are 
ordered by the size of their discriminants
the
fraction 
of them
that satisfy $$\rank_4(\cl(K))=j_1, \ \rank_4(\cl(K,c))=j_2
$$
approaches 
$$ 
\frac{\eta_{\infty}(2)}{\eta_{j_1}(2)^2   2^{j_1^2}}
\frac{\#\{\phi \in \Hom_{\mathbb{F}_2}(\mathbb{F}_2^{j_1},\mathbb{F}_2^{\#S_3}): \rank(\phi)=\#S-(j_2-j_1)\}}{\#\Hom_{\mathbb{F}_2}(\mathbb{F}_2^{j_1},\mathbb{F}_2^{\#S_3})}
.$$
\end{theorem}
For 
$M \in \mathbb{Z}_{\geq 1} $
and 
$s \in \mathbb{Z}_{\geq 1} \cup \{\infty\}$,
$\eta_s(M)$
denotes
$\prod_{i=1}^{s}(1-M^{-i})$.
For the statement in full generality see Theorem ~\ref{theorem: joint distribution of 4-ranks, explicit version}.

The special
case $c=1$ 
of
Theorem~\ref{thm 4rk:special case}
recovers a result
of Fouvry and Kl\"uners ~\cite[Cor. 1]{MR2282914} (in the subfamily of imaginary quadratic number fields above). 
The theorem of Fouvry and Kl\"uners on $4$-ranks is one of the strongest pieces of evidence for the heuristic of 
Cohen--Lenstra and Gerth about the distribution of the $p$-Sylow subgroup of the class group of an imaginary quadratic number field.

Indeed, for odd primes $p$, 
Cohen and Lenstra~\cite{MR756082} constructed a heuristic model to predict the outcome of any statistic on the $p$-Sylow of the class group of imaginary quadratic number fields. For every prime $p$
they equipped the set of isomorphism classes of abelian $p$-groups, $\mathcal{G}_p$, with the only probability measure that gives to each abelian $p$-group $G$ a weight inversely proportional to $\#\aut(G)$. This measure is now often called the Cohen--Lenstra measure on $\mathcal{G}_p$, and denoted by $\mcl$. Their heuristic model, for odd primes $p$, consisted in
predicting the equidistribution of $\cl(K)[p^{\infty}]$ in $\mathcal{G}_p$, as $K$ 
ranges
through natural families of imaginary quadratic number fields. Later, Gerth~\cite{MR887792}
adapted this heuristic model for $p=2$. 
His idea was that the only obstruction for $\cl(K)[2^{\infty}]$ to behave like a random abelian $2$-group in the sense of Cohen--Lenstra 
comes from $\cl(K)[2]$; therefore his heuristic model is that $2\cl(K)[2^{\infty}]$ behaves like a random abelian $2$-group. The result of Fouvry and Kl\"uners can then be formulated by saying that, consistently with Gerth's conjecture, the $2$-torsion of $2\cl(K)$ behaves like the $2$-torsion of a random abelian $2$-group in the sense of Cohen--Lenstra.

Before the present paper, no analogue 
of any of these heuristics 
has been proposed
for ray class groups. 
Our second main achievement, aside from the proof of Theorem~\ref{thm 4rk:special case},
is to provide an extension of the Cohen--Lenstra and Gerth heuristics for
ray class groups.  
Theorem~\ref{thm 4rk:special case} will then be the simplest evidence supporting our new heuristic for ray class groups. 
In particular, we provide the conjectural analogue of Theorem~\ref{thm 4rk:special case} for all odd primes $p$. 
Partition $S$ into $S_1 \cup \ldots \cup S_{p-1}$, where $l \in S_i$ if $l  \equiv  i \ \md{p} $. 
\begin{conjecture}
Let $p$ be an odd prime.
Consider all
imaginary quadratic number fields $K$
having
the property
$\mathcal{O}_K/c \cong_{\emph{\text{ring}}} \prod_{l \in S} \mathbb{F}_{l^2}$. 
When such $K$ are 
ordered by the size of their discriminants
the
fraction 
of them
that satisfy $$\rank_p(\cl(K))=j_1, \ \rank_p(\cl(K,c))=j_2
$$
approaches  
$$ 
\frac{\eta_{\infty}(p)}{\eta_{j_1}(p)^2   p^{j_1^p}}
 \frac{\#\{\phi \in \Hom_{\mathbb{F}_p}(\mathbb{F}_p^{j_1},\mathbb{F}_p^{\#S_{p-1}}): \rank(\phi)=\#S_1+\#S_{p-1}-(j_2-j_1)\}}
{\#\Hom_{\mathbb{F}_p}(\mathbb{F}_p^{j_1},\mathbb{F}_p^{\#S_{p-1}})}
.$$
\end{conjecture}
For the statement in the
general case 
see Conjecture~\ref{c:wc},
in particular,
in the main body of the paper,
we shall allow
any admissible
ring structure for
$\mathcal{O}_K/c$. 
From our model in
its full generality
we shall
derive conjectural formulas for the average size of the $p$-torsion of ray class groups of imaginary quadratic number fields.
\begin{conjecture}
\label{conj:7}
Let $p$ be an odd prime.
The average value of $\#\cl(K,c)[p]$ as $K$ ranges over imaginary quadratic
number fields
with $\gcd(\disc(K),c)=1$ 
and
ordered by their discriminant
is: \\
(1) 
$$p^{\#\{l \emph{ prime}: \ l|c, l \equiv  1 \md{p}\}} \Big(1+\Big(\frac{p+1}{2}
\Big)^{\!\#\{l \emph{ prime}: \ l|c, l  \equiv  1  \text{ or } -1 \md{p}\}}\Big)$$ 
if $p^2$ does not divide $c$, \\
(2)  $$p^{\#\{l \emph{ prime}: \ l|c, l \equiv  1 \md{p}\}+1} \Big(1+p\Big(\frac{p+1}{2}
\Big)^{\!\#\{l \emph{ prime}: \ l|c, l \equiv  1  \text{ or } -1 \md{p}\}}\Big)
$$  if $p^2$ divides $c$.
\end{conjecture}
 
For $p=3$ this conjecture was 
recently
proved by Varma~\cite{arXiv:1609.02292}
using geometry of numbers.
In~\cite[\S 1]{arXiv:1609.02292}
she 
asked whether one can 
formulate an extension of the Cohen--Lenstra heuristic that explains her result. 
Our model for ray class groups 
settles this
for imaginary quadratic number fields 
(for the full comparison with Varma's result see \S \ref{s:var}).

Our main theorems and conjectures are not merely about the group $\cl(K,c)$
but also about the entire exact sequence naturally attached to it:
$$1 \to \frac{(\mathcal{O}_K/c)^{*}}{\mathcal{O}_K^{*}} \to \cl(K,c) \to \cl(K) \to 1.$$
For simplicity, in this section we will continue to assume that all the primes in $S$ are inert in $K$. Then one can show that there is a long exact sequence whose first terms are
$$ 1 \to \Big(\frac{(\mathcal{O}_K/c)^{*}}{\langle -1 \rangle}\Big)^2[2] \to (2\cl(K,c))[2] \to (2\cl(K))[2] \overset{\delta_2(K)}\to \prod_{l \in S_3}\frac{\mathbb{F}_{l^2}^{*2}}{\mathbb{F}_{l^2}^{*4}}
.$$
To obtain the last map one chooses any identification between  
$\frac{\big(\frac{(\mathcal{O}_K/c)^{*}}{\langle -1\rangle}\big)^2}
{\big(\frac{(\mathcal{O}_K/c)^{*}}{\langle -1\rangle}\big)^4}$ 
and $\prod_{l \in S} \frac{\mathbb{F}_{l^2}^{*2}}{\mathbb{F}_{l^2}^{*4}}$ via an identification of the rings $\mathcal{O}_K/c$ and  $\prod_{l\in S}\mathbb{F}_{l^2}$. The resulting set of maps is an orbit under  $\aut_{\text{ring}}(\prod_{l\in S}\mathbb{F}_{l^2})$, acting by post-composition. But  $\aut_{\text{ring}}(\prod_{l\in S}\mathbb{F}_{l^2})$ acts trivially on $\prod_{l \in S_3}\frac{\mathbb{F}_{l^2}^{*2}}{\mathbb{F}_{l^2}^{*4}}$, so one has a canonical identification.

Let $Y$ be a subspace of $\prod_{l \in S_3}\frac{\mathbb{F}_{l^2}^{*2}}{\mathbb{F}_{l^2}^{*4}}$ and $j$ a non-negative integer.  
In this setting we manage to control
the statistical distribution of 
$(\#2\cl(K))[2],\Im(\delta_2(K))$,
thus providing a considerable 
refinement of Theorem~\ref{thm 4rk:special case}.
Our result is as follows.
\begin{theorem}\label{main thm intro}
Consider all
imaginary quadratic number fields $K$
such that $\disc(K) \equiv  1  \md{4}$
and
$\mathcal{O}_K/c \cong_{\emph{\text{ring}}} \prod_{l \in S} \mathbb{F}_{l^2}$. 
When such $K$ are 
ordered by the size of their discriminants
the
fraction 
of them
that satisfy 
$$(2\cl(K))[2] \cong \mathbb{F}_2^{j}, \ \Im(\delta_2(K))=Y$$
approaches $$\frac{\eta_{\infty}(2)}{\eta_{j_1}(2)^2  2^{{j_1}^2}}
\frac{\#\epi_{\mathbb{F}_2}(\mathbb{F}_2^j,Y)}{\#\Hom_{\mathbb{F}_2}
\Big(
\mathbb{F}_2^j,
\prod_{l \in S_3}
\frac{\mathbb{F}_{l^2}^{*2}}{\mathbb{F}_{l^2}^{*4}}
\Big)
}.$$
\end{theorem}
This means that 
$(\#(2\cl(K))[2],\Im(\delta_2(K)))$ behaves like
$(\#G[2],\Im(\delta))$, where $G$ is a random abelian $2$-group 
in the
Cohen--Lenstra sense, and $\delta:G[2] \to \mathbb{F}_2^{\#S_3}$ is a random map.
For the statement in full generality
see Theorem~\ref{theorem:distribution of delta maps}. 
We show in~\S\ref{section:heuristic at 2} that this result is also predicted by our heuristic model. Our model enables us to provide a conjectural analogue of Theorem~\ref{main thm intro} for all odd $p$. Its formulation is in Conjecture ~\ref{conj:3}.

Theorem~\ref{main thm intro} determines the joint distribution of the pair $(\#(2\cl(K))[2],\Im(\delta_2(K)))$. 
Theorem~\cite[Cor.1]{MR2282914} of Fouvry and Kl\"uners determines the distribution of the first component, $\#(2\cl(K))[2]$
via the use of
another result of the two authors,~\cite[Theorem 3]{MR2276261}, 
where they obtained asymptotics for all moments of
$\#(2\cl(K))[2]$.
A surprising feature of our work is that we establish the joint distribution of the pair $(\#(2\cl(K))[2],\Im(\delta_2(K)))$ by means of the moment-method, despite the fact that $\Im(\delta_2(K))$ is  
not a number. 
\emph{
Although the general philosophy of using 
moments to study 
distributions is standard in the literature related to the Cohen--Lenstra heuristics}(\emph{see, for example,~\cite{arxiv.org/abs/1504.04391}}),
\emph{we stress that 
no 
object like the image of the $\delta$-map has
been treated
in the subject.}
It is instructive to see how we
incorporate
the image-data
into the
Fouvry--Kl\"{u}ners method.
We do this by introducing for every real character $\chi: \prod_{l \in S_3}\mathbb{F}_{l^2}^{*2} \to \mathbb{R}^{*}$, the random variable 
$$m_{\chi}(\delta_2(K)):=\#\text{ker}(\chi(\delta_2(K))).$$
To know the pair $(\#(2\cl(K))[2],\Im(\delta_2(K)))$ is equivalent to 
knowing $(m_{\chi}(\delta_2(K)))_{\chi}$. 
However,
the advantage is that the latter is a \emph{numerical} vector
and therefore one can hope to apply the
method of moments
to control its distribution. This is precisely what we achieve in Theorem~\ref{m.t. on multi-moments}. The expressions that 
appear during the proof of Theorem~\ref{m.t. on multi-moments}
are of the shape
$$ \sum_{D<X} 
\prod_{\chi}m_{\chi}(\delta_2(\Q(\sqrt{-D})))^{k_{\chi}}
,$$ 
where $D$ ranges over all positive
square-free integers with $D  \equiv  3  \md{4} $
and
$\chi$
ranges
over
all
real characters $\chi: \prod_{l \in S_3}\mathbb{F}_{l^2}^{*} \to \mathbb{R}^{*}$.
As explained in \S\ref{pre-indexing}, the additional complexity of these expressions
compared to the classical case settled by Fouvry and Kl\"uners, is tempered by the fact that, with our heuristic model for ray class groups, we already have a candidate main term. In particular, the shape of its expression suggests a way to sub-divide the sum, with the benefit of hindsight, in many smaller sub-sums. For each of these sub-sums it turns out that the techniques of Fouvry and Kl\"uners are applicable
with only minor modifications.
After proving
Theorem~\ref{m.t. on multi-moments}
we turn our attention to the 
distribution 
of
$(\#(2\cl(K))[2],\Im(\delta_2(K)))$,
which we reconstruct
from the mixed moments by
following an argument of
Heath-Brown~\cite{MR1292115}.  

We stress that Theorem~\ref{main thm intro} is stronger than Theorem~\ref{thm 4rk:special case}. 
Here the finer information (which is the \emph{image} of the $\delta$-map), is obtained precisely owing to the fact that we use ring identifications rather than merely group identifications\footnote{We thank Hendrik Lenstra for having suggested this.}.
Using the latter we could have studied only the \emph{size} of $\Im(\delta_2(K))$, which is precisely what occurs
in Theorem~\ref{thm 4rk:special case}. 
On the other hand, it is important to note 
that the techniques employed in the proof of Theorem~\ref{main thm intro} are not applicable 
in studying directly the moments of the isolated quantity $\#(2\cl(K,c))[2]$: we can access the 
distribution of the quantity $\#(2\cl(K,c))[2]$
only by the moments of a finer object, the $\delta$-map. 
This contrast reflects the fact that the natural algebraic structure attached to the ray class group is the entire
exact sequence naturally attached to it, rather than just the isolated group $\cl(K,c)$. 
It is precisely this phenomenon that leads us to formulate a general heuristic for ray class \emph{sequences} of conductor $c$. 
In this framework, Theorem ~\ref{main thm intro} gives compelling evidence that our heuristic model predicts correct answers also when it is challenged to produce the outcome of statistics about the \emph{ray class sequence}, and not only when, less directly, one isolates the group $\cl(K,c)$.

Encouraged by this corroboration, we formulate our heuristic to predict the outcome of \emph{any} statistical question about the $p$-part of the ray class sequence, viewed as an exact sequence of Galois modules. A positive side effect of this enhanced generality is the consequent logical simplification of our conjectural framework: our heuristic is based on a simple unifying principle, which, if true, implies at once all our conjectures. This heuristic principle is stated in \S\ref{s:2} for an odd prime $p$, and in \S\ref{section:heuristic at 2} for $p=2$.

Let $p$ be a prime and $G$ a finite abelian $p$-group. The following is an attractive and easy example of the conjectural conclusions that are available in this new model:
\begin{conjecture}
Consider all
imaginary quadratic number fields $K$
having the property that
$\mathcal{O}_K/c \cong_{\emph{\text{ring}}} \prod_{l \in S} \mathbb{F}_{l^2}$. 
When such $K$ are 
ordered by the size of their discriminants,
the
fraction 
of them
having the properties that 
the $p$-part of the ray class sequence of modulus $c$ splits
and
$$\cl(K)[p^{\infty}] \cong_{\emph{ab.gr.}} G,$$
approaches 
$$ \frac{\eta_{\infty}(p)}{\#\aut_{\emph{ab.gr.}}(G)}
 \frac{1}{\#\Hom_{\emph{ab.gr.}}(G,\prod_{l \in S_{p-1}}\mathbb{F}_{l^2}^{*})}.$$ 
\end{conjecture}

\subsection{Comparison with the literature}
The present work sits in an active area of research focused on extending the classical Cohen--Lenstra heuristics to other interesting arithmetical objects and on establishing the correctness of these statistical models
in cases where an `analytically-friendly'
description of the problem is available. 
Developments along this line of research can be found in the very recent work by Wood~\cite{arXiv:1702.04644}, which provides a heuristic for the average number of unramified $G$-extensions of a quadratic number field for any finite group $G$: the Cohen--Lenstra heuristics are recovered by taking $G$ to be an abelian group. It would be interesting to reach the generality of both the present paper and \cite{arXiv:1702.04644}, by considering $G$-extensions with prescribed ramification data. The evidence provided in ~\cite{arXiv:1702.04644} is over function fields, by means of the approach
of Ellenberg, Venkatesh and Westerland~\cite{MR3488737}. 
In a recent preprint, Alberts and Klys~\cite{arXiv:1611.05595}
offered evidence for the heuristics in Wood's work~\cite{arXiv:1702.04644} over number fields
using the approach of Fouvry and Kl\"uners. It is interesting to note that in a previous work
Klys~\cite{arXiv:1610.00226}
extended the work of Fouvry and Kl\"uners to the $p$-torsion of cyclic degree $p$ extensions. 
These last two examples, together with the present work, show the remarkable versatility of the method used in~\cite{MR2276261}
and pioneered (in the context of Selmer groups)
by Heath-Brown~\cite{MR1292115}.

The case of narrow class groups was investigated by Bhargava and Varma~\cite{MR3369305} and by 
Dummit and Voight~\cite{arXiv:1702.00092}. 
The latter work provides, among other things, a conjectural formula for the average size of the $2$-torsion of narrow class groups among the family of $S_n$-number fields, for odd $n$. For $n=3$, this was a theorem of Bhargava and Varma~\cite{MR3369305}.

Very recently,
Jordan, Klagsbrun, Poonen, Skinner and Zaytman~\cite{arXiv:1703.00108}
made a conjecture for the distribution of the $p$-torsion of $K$-groups of real and imaginary quadratic number fields.  
Building on the recent improvement of the work of Bhargava, Shankar and Tsimerman~\cite{MR3090184}, 
they established their conjecture for the average size of the $3$-torsion. Incidentally, the work~\cite{MR3090184} is also
employed by Varma~\cite{arXiv:1609.02292} on the average $3$-torsion of ray class groups, 
which is placed in a general conjectural framework by the present paper.

Despite this rich context of developments, the present paper is, to the best of our knowledge, the first one to propose a heuristic model for the ray class sequence of imaginary quadratic number fields  and to prove its correctness for the pair $(\#(2\cl(K))[2],\Im(\delta_2(K)))$, establishing, as a corollary, the joint distribution of the $4$-ranks of $\cl(K)$ and $\cl(K,c)$. 
\subsection{Organization of the material}
The remainder of this paper is organized as follows: 
In \S\ref{s:2} we explain our heuristic model for the distribution of the $p$-part of ray class sequences of imaginary quadratic number fields, for odd primes $p$.  We draw several conjectures from this heuristic principle and verify its consistency with the theorems of
Varma~\cite{arXiv:1609.02292} in the imaginary quadratic case.

In \S\ref{section:heuristic at 2} we examine the case $p=2$. This case requires some additional work to isolate the `random' 
part of the $2$-Sylow of the ray class sequences of imaginary quadratic number fields. 
This additional difficulty arises already for the ordinary class group as can be seen in the work of
Gerth~\cite{MR887792}. However,
for ray class sequences 
overcoming such difficulties
is much
more intricate 
due to the more articulate underlying 
algebraic structures.
This will allow us to 
formulate a number of predictions that will be
proved
in \S\S\ref{section:theorems on 2-parts of ray sequences}-\ref{from m.m. to distr.}.
A key step 
in these proofs
is the reformulation of the problem about $4$-ranks into a purely analytic problem about mixed moments.
For this we introduce the notion of special divisors in \S\ref{section:special divisors}
and certain related statistical questions that will be subsequently
answered. This statistic is a special case of a ray class group statistic, as subsequently 
established in \S\ref{section:theorems on 2-parts of ray sequences}. 
Therefore the material of \S\ref{section:heuristic at 2} would implicitly provide a heuristic for it. Nevertheless, 
in \S\ref{section:special divisors}
we present the problem and the heuristic in 
a direct way using the language of special divisors.
This has the advantage that \S\ref{section:special divisors},
Theorems~\ref{m.t. on multi-moments}-\ref{m.t. on distribution},
 \S\ref{section: special divisors thm} and \S\ref{from m.m. to distr.}
are mostly analytic in nature and
can be read independently of the algebraic considerations
in \S{\ref{s:2} and \S\ref{section:heuristic at 2}.

In \S\ref{section:theorems on 2-parts of ray sequences} 
we
state the main theorems about the $2$-part of the ray class sequences and reduce
their proof so as to establish the predictions in 
\S\ref{section:special divisors}. The section ends with the statement of the corresponding main theorems on special divisors. 
In \S\ref{section: special divisors thm} we prove the main theorem on mixed moments attached to the maps on special divisors introduced in 
\S\ref{section:special divisors}.
Finally, in \S\ref{from m.m. to distr.} we reconstruct the distribution from the mixed moments, concluding the proof of all 
theorems 
stated in \S\ref{section:theorems on 2-parts of ray sequences}.  
\begin{notation}
The symbol $\disc(K)$ will always 
refer to the discriminant of a number field $K$.
Let us furthermore denote 
\begin{equation}
\label{eq:imagin}
\c{F}
:=\{K \text{ imaginary quadratic number field}\}.
\end{equation} 
\end{notation}

\begin{acknowledgements}
We are very grateful to Hendrik Lenstra for several insightful discussions
and for useful feedback during the course of this project.
In particular, we thank him for 
suggesting to consider
the first terms of the ray class sequences only up to \emph{ring} automorphisms, 
which turned out to be a natural level of greater generality where we could prove our main theorems on $4$-ranks. 
We thank Alex Bartel for many stimulating discussions about our work,
as well as 
organizing an inspiring
conference on the Cohen--Lenstra heuristics in Warwick in July $2016$, where this project started. 
We also wish to thank
Djordjo Milovic and Peter Koymans for useful discussions
and
Ila Varma
and 
Peter Stevenhagen for
profitable feedback.
Furthermore,
we thank
Alex Bartel,
Joseph Gunther and Peter Koymans
for helpful remarks on earlier versions of this
paper.
\end{acknowledgements}

\section{Heuristics and conjectures for $p$ odd} 
\label{s:2}
Let $p$ be an odd prime number and $c$ a positive integer. Denote by $C_2$ a group with $2$ elements and denote by $\tau$ its generator.
In this section we provide
a heuristic model that predicts the statistical behavior of the exact sequence 
of $\mathbb{Z}_p[C_2]$-modules attached to the ray class group of conductor $c$ of an imaginary quadratic number field $K$.
Denote it by
\begin{equation}
\label{eq:exact}
S_p(K):=
\Big(
1 \to \frac{(\mathcal{O}_K/c)^{*}}{\mathcal{O}_K^{*}}[p^{\infty}] \to \cl(K,c)[p^{\infty}] \to \cl(K)[p^{\infty}] \to 1
\Big)
,\end{equation}
where
the 
$C_2$-action
comes from the natural action of
$\Gal(K/\mathbb{Q})$
on each term of the sequence.
The reader is referred to~\cite[\S IV]{neuki}
for related background material.
We  shall call $S_p(K)$ the $p$-part of the \textit{ray class sequence} of conductor $c$. 
We shall henceforth ignore  
the fields 
$K=\mathbb{Q}(\text{i})$ and $K=\mathbb{Q}(\sqrt{-3})$,
to ensure that 
$\mathcal{O}_K^{*}=\langle -1 \rangle$.
Owing to $p\neq 2$ we 
furthermore have
$
((\mathcal{O}_K/c)^{*}/\langle -1 \rangle)[p^{\infty}]
=
(\mathcal{O}_K/c)^{*}[p^{\infty}] 
$,
thus allowing us to write
$$S_p(K):=(1 \to (\mathcal{O}_K/c)^{*}[p^{\infty}] \to \cl(K,c)[p^{\infty}] \to \cl(K)[p^{\infty}] \to 1)
.$$
Denote by $\mathcal{G}_p$ a set of representatives of isomorphism classes of finite abelian $p$-groups, viewed as $C_2$-modules 
under the action of $-\Id$
and 
call $G_p(K)$ the unique representative of $\cl(K)[p^{\infty}]$ in $\mathcal{G}_p$.
Any family of imaginary quadratic fields 
can be partitioned 
in finitely many subfamilies where the isomorphism class of the ring 
$\mathcal{O}_K/c$ is fixed, 
by imposing finitely many congruence conditions on the discriminants.
Therefore
we can always assume that 
$(\mathcal{O}_K/c)^{*}$ has been fixed as the unit group of some ring that is independent of $K$.

\begin{definition}
\label{def:oftype}
Let $K,c$
be as above and $R$ a finite commutative ring.
We shall say that $K$ is of type $R$ if 
$\mathcal{O}_K/\text{char}(R) \cong R$
as rings.
With this definition in mind let us denote 
\begin{equation}
\label{eq:fff}
\c{F}(R)
:=\{K \text{ imaginary quadratic number field of type } R \}.
\end{equation}
\end{definition}
From now on we will assume that $R$ is of the form $R:=\mathcal{O}_{\mathcal{A}}/c$, where $\mathcal{O}_{\mathcal{A}}$ is the integral closure of $\prod_{l|c}\mathbb{Z}_l$ in $\mathcal{A}:=\prod_{l|c}E_l$, with $E_l$ being an etale $\mathbb{Q}_l$-algebra of degree $2$. Under this assumption, a positive fraction of all discriminants lies in $\c{F}(R)$.

Suppose $K$ is of type $R$. Then  
$(\mathcal{O}_K/c)^{*}$ can be identified with $R^{*}$ via any restriction of a ring isomorphism, that is via any element of 
$\isom_{\text{ring}}(\mathcal{O}_K/c,R)$. 
Furthermore, 
we can identify $\cl(K)[p^{\infty}]$ and $G_p(K)$  via any element of 
$\isom_{\text{ab.gr.}}(\cl(K)[p^{\infty}],G_p(K))$.
Therefore applying $\isom_{\text{ring}}(\mathcal{O}_K/c,R) \times \isom_{\text{ab.gr.}}(\cl(K)[p^{\infty}],G_p(K))$ to $S_p(K)$, we obtain a unique orbit 
$$ O_{c,p}(K) \in 
\ext_{\mathbb{Z}_p[C_2]}(G_p(K),R^{*}[p^{\infty}])/(\aut_{\text{ring}}(R) \times \aut_{\text{ab.gr.}}(G_p(K))).$$
 We refer the reader to~\cite[\S 3]{weibel} 
for definition and properties of $\ext_{S}(A,B)$, where $S$ is a ring and $A,B$ are $S$-modules. For the remainder of the paper, given $S$-modules $A,B,C,A',B'$ and $C'$, we call a \emph{commutative diagram} of $S$-modules, a diagram of maps of $S$-modules
\[ 
\begin{array}{ccc}  0 &\overset{}\to& B_1  \underset{f_1}\to  \\  && 
\downarrow \psi_1 
\\ 0 &\underset{}\to& B_2 \underset{f_2}\to  \end{array} \begin{array}{ccc}  C_1 &\underset{g_1}\to& A_1 \to 0 
\\ \downarrow \psi_2 &&
 \downarrow \psi_3\\ 
C_2 &\underset{g_2}\to& A_2 \to 0, \end{array}
\]
with $\psi_2 \circ f_1=f_2 \circ \psi_1$ and $\psi_3 \circ g_1=g_2 \circ \psi_2$.
Note  that $\cl(K_1)[p^{\infty}] \cong_{\text{ab.gr}} \cl(K_2)[p^{\infty}]$ and 
$O_{c,p}(K_1)=O_{c,p}(K_2)$ if and only if there is a commutative diagram of $\mathbb{Z}_p[C_2]$ modules
\[ 
\begin{array}{ccc}  0 &\overset{}\to& (\mathcal{O}_{K_1}/c)^{*}[p^{\infty}]  \to  \\  && 
\downarrow \phi_1 
\\ 0 &\underset{}\to& (\mathcal{O}_{K_2}/c)^{*}[p^{\infty}] \to  \end{array} \begin{array}{ccc}  \cl(K_1,c)[p^{\infty}] &\overset{}\to& \cl(K_1)[p^{\infty}] \to 0 
\\ \downarrow \phi_2 &&
 \downarrow \phi_3\\ 
\cl(K_2,c)[p^{\infty}] &\underset{}\to& \cl(K_2)[p^{\infty}] \to 0, \end{array}
\]
with $\phi_1$ being the restriction of a ring isomorphism and $\phi_3$ being an isomorphism of abelian groups.  
\begin{definition}
\label{def:spaceofexactseq}
Define 
$\mathcal{S}_p(R)$ as
the set of equivalence classes of pairs $(G,\theta)$, where 
$$G \in \mathcal{G}_p, \
\theta \in \ext_{\mathbb{Z}_p[C_2]}(G,R^{*}[p^{\infty}])$$
under the following equivalence relation:
two pairs $(G_1,\theta_1), (G_2,\theta_2)$ are identified if $G_1=G_2$ and $\theta_1$ and $\theta_2$ are in the same $\aut_{\text{ring}}(R) \times \aut_{\text{ab.gr.}}(G_1)$-orbit. 
\end{definition}
Let us  
denote by
$\widetilde{\mathcal{S}}_p(R)$ the set of pairs $(G,\theta)$ where $G \in \mathcal{G}_p$ and $\theta \in \ext_{\mathbb{Z}_p[C_2]}(G,R^{*}[p^{\infty}])$, 
thus bringing into play
the quotient map
$ \pi:\widetilde{\mathcal{S}}_p(R) \to \mathcal{S}_p(R)$.
We are interested in studying the distribution of   $S'_p(K)$
given by the pair
$$ K \mapsto S'_p(K):=(G_p(K),O_{c,p}(K)) \in \mathcal{S}_p(R).$$ 
\begin{definition}
\label{def:mcl}
Let $\mcl$ be the unique probability  measure on $\mathcal{G}_{p}$ which gives to each abelian $p$-group $G$ 
a weight inversely proportional to the size of the automorphism group of $G$.
\end{definition}
This measure was introduced by Cohen and Lenstra in ~\cite{MR756082} to predict the distribution of 
$G_p(K)$,
the first component 
of $S'_p(K)$. We shall introduce a measure on $\mathcal{S}_p(R)$
that enables us to predict the joint 
distribution of the vector $S'_p(K)$.
Consider the discrete $\sigma$-algebra
on both
$\widetilde{\mathcal{S}}_p(R), \mathcal{S}_p(R)$  
and 
equip $\widetilde{\mathcal{S}}_p(R)$ with 
the following
measure,
\begin{equation}
\label{def:sec}
\widetilde{\mu}_{\text{seq}}((G,\theta)):=
\frac{\mcl(G)}{\#\ext_{\mathbb{Z}_p[C_2]}(G,R^{*}[p^{\infty}])}.
\end{equation}
Let $\mseq:=\pi_*(\widetilde{\mu}_{\text{seq}})$ 
be the  pushforward measure of  $\widetilde{\mu}_{\text{seq}}$ 
on $\mathcal{S}_p(R)$ via $\pi$.
It is evident that  $\widetilde{\mu}_{\text{seq}}$ and $\mseq$ are probability measures.
We now 
formulate a
heuristic which roughly states 
that 
ray class sequences 
equidistribute
within the set of isomorphism classes of exact sequences
with respect to the measure $\mseq$. 
\begin{heuristic}
\label{h:lebkov}
For any `reasonable' function $f:\mathcal{S}_p(R) \to \mathbb{R}$ we have  
$$\lim_{X \to \infty}
\#\{K \in \c{F}(R) : |\disc(K)|\leq X\}^{-1}
\sum\limits_{\substack{K \in \c{F}(R) \\|\disc(K)|\leq X}}f(S'_p(K))
=\sum_{S \in \mathcal{S}_p(R)}f(S)\mseq(S).$$
\end{heuristic}
Letting $f$ be the indicator function of a
singleton
yields the following statement.
\begin{conjecture}
\label{conj:1}
For any   $S \in \mathcal{S}_p(R)$
we have 
$$ 
\lim_{X \to \infty}
\frac{\#\{K \in \c{F}(R) : |\disc(K)|\leq X,S'_p(K)=S\}}
{\#\{K \in \c{F}(R) : |\disc(K)|\leq X\}}
=\mseq(S)
.$$
\end{conjecture} 
A special concrete example
is the case of split sequences.
\begin{conjecture}
\label{con:eibaar}
The fraction of $K \in \c{F}(R)$, ordered by the size of their discriminant,
for which $\cl(K)[p^{\infty}] \cong_{\mathrm{ab.gr.}} G$ and the $p$-part of the ray class sequence of modulus $c$ splits, approaches
$$ 
\frac{\mu_{\emph{CL}}(G)}{\#\Hom_{\mathrm{ab.gr.}}(G,R^{*}[p^{\infty}]^{-})},$$
where 
$(R^{*}[p^{\infty}])^{-}$ 
denotes
the minus part 
of
$R^{*}[p^{\infty}]$ 
under the action of $C_2$.
\end{conjecture}
Indeed, 
$\ext_{\mathbb{Z}_p[C_2]}(G,R^{*}[p^{\infty}])=\ext_{\mathbb{Z}_p}(G,(R^{*}[p^{\infty}])^{-})$
holds,
hence
Conjecture~\ref{con:eibaar} is 
derived from Conjecture~\ref{conj:1}
by recalling that for two finite abelian $p$-groups $A,B$, there is a non-canonical isomorphism 
$\ext_{\mathbb{Z}_p}(A,B) \cong_{\mathrm{ab.gr.}}\Hom_{\mathbb{Z}_p}(A,B)$.

\subsection{Conjectures on the $p$-torsion}
\label{s:ptors}
We next state certain consequences of Heuristic assumption~\ref{h:lebkov}
regarding 
the $p$-torsion of the ray class sequences.
Taking $p$-torsion in~\eqref{eq:exact}
provides us 
with 
a long exact sequence whose first four terms are given by 
$$S(K)[p]:=
\Bigg(
1 
\to (\mathcal{O}_K/c)^{*}[p] \to \cl(K,c)[p] \to \cl(K)[p] 
\xrightarrow[]{\delta_p(K)}
 \frac{(\mathcal{O}_K/c)^{*}}{((\mathcal{O}_K/c)^{*})^{p}}
\Bigg)
,$$
where the map
$ \delta_p(K)$
is defined 
as follows:
given a class $x \in \cl(K)[p]$
pick a representative ideal 
$\mathcal{I}$ of $x$
which is coprime to $c$,
take a generator of $\mathcal{I}^{p}$ and reduce it
modulo $c$. The choice of another representative does not change it modulo $p$-th powers. More generally, taking $p$-torsion in any short exact sequence of $\mathbb{Z}_p[C_2]$-modules
$$S:=(0 \to A \to B \to C \to 0)
$$
provides us with a long exact sequence whose first terms are 
$$S[p]:=
\Bigg(
1 
\to A[p] \to B[p] \to C[p] 
\xrightarrow[]{\delta_p(S)}
 \frac{A}{pA}
\Bigg),
$$
where $\delta_p(S)$ is defined in the same way as explained above (in particular we have $\delta_p(S_p(K))=\delta_p(K)$). Thus this provides a map sending an element $\theta$ of $\ext_{\mathbb{Z}_p[C_2]}(C,A)$ to a map $\delta_p(\theta):C[p] \to A/pA$. We will make repeatedly use of the following fact.
\begin{proposition} \label{induces surj map}
The map sending $\theta$ to $\delta_p(\theta)$, from $\ext_{\mathbb{Z}_p[C_2]}(C,A)$ to $\Hom_{\mathbb{Z}_p[C_2]}(C[p],A/pA)$, is a surjective group homomorphism.
\end{proposition}
The reader interested in a proof of Proposition \ref{induces surj map}, can look at the proof of the analogous, but more complicated, Proposition \ref{delta maps for ext tilde}: all the ingredients for the proof of Proposition \ref{induces surj map} are contained in the proof of Proposition \ref{delta maps for ext tilde}.

Next we shall define $j:=\dim_{\mathbb{F}_p}(\cl(K)[p])$
and apply any pair of identifications from 
$\isom_{\mathbb{F}_p}(\cl(K)[p],\mathbb{F}_p^j) \times \isom_{\text{ring}}(\mathcal{O}_K/c,R)$. 
Therefore, 
we obtain a unique orbit of maps $\phi \in \Hom_{\mathbb{F}_p}(\mathbb{F}_p^j,(\frac{R^{*}}{R^{*p}})^{-})$ under the action of $\text{GL}_{j}(\mathbb{F}_p) \times \aut_{\text{ring}}(R)$. This is tantamount to having a $\aut_{\text{ring}}(R)$-orbit of images in $(\frac{R^{*}}{R^{*p}})^{-}$ of $\delta_p(K)$ via any of the previous identifications. We denote this orbit by $[\Im(\delta_p(K))]$. The assignment $K \mapsto [\Im(\delta_p(K))]$ attaches to each imaginary quadratic field $K \in \mathcal{F}_c(R)$ 
a well-defined $\aut_{\text{ring}}(R)$-orbit of vector sub-spaces of $(\frac{R^{*}}{R^{*p}})^{-}$.

By Proposition \ref{induces surj map}, the map 
$$\ext_{\mathbb{Z}_p}(G,R^{*}[p^{\infty}]^{-}) \to \Hom_{\mathbb{Z}_p}(G[p],(R^{*}/R^{*p})^-)
$$ induces, by pushforward, the 
counting probability 
measure from $\ext_{\mathbb{Z}_p}(G,(R^{*}[p^{\infty}])^{-}) $ to $\Hom_{\mathbb{Z}_p}(G[p],(R^{*}/R^{*p})^-)$. 
Therefore, fixing
 a sub-$\mathbb{F}_p$-space $Y$
of $(\frac{R^{*}}{R^{*p}})^{-}$
and a non-negative integer $j$,
Heuristic assumption~\ref{h:lebkov}
supplies us with the following. 
\begin{conjecture}
\label{conj:3}
The proportion of $K \in \c{F}(R)$ ordered by the size of their discriminant, 
for which $\dim_{\mathbb{F}_p}(\cl(K)[p])=j$ and $[\Im(\delta_p(K))]$ is $O(Y)$, the $\aut_{\emph{ring}}(R)$-orbit of $Y$, approaches 
$$ \mu_{\emph{CL}}(G \in \mathcal{G}_p:\dim_{\mathbb{F}_p}(G[p])=j) 
\
 \frac{\#\epi_{\mathbb{F}_p}(\mathbb{F}_p^j,Y)\cdot \#O(Y)}{\#\Hom_{\mathbb{F}_p}
\big(\mathbb{F}_p^j,
(R^{*}/R^{*p})^{-}\big)
}
.$$
\end{conjecture}
We will prove the analogous statement of this
Conjecture~\ref{conj:3} for $p=2$ in Theorem  ~\ref{theorem:distribution of delta maps}.
A concrete special case is given by the following
\begin{conjecture} 
\label{conj:icecream}
The proportion of $K \in \c{F}(R)$ ordered by the size of their discriminant,
for which 
$\dim_{\mathbb{F}_p}(\cl(K)[p])=j$ and 
$\cl(K,c)[p]$
splits as the direct sum of $\cl(K)[p]$ and $(\mathcal{O}_K/c)^{*}[p]$, approaches 
$$  
\frac{\mu_{\emph{CL}}(G \in \mathcal{G}_p:\dim_{\mathbb{F}_p}(G[p])=j)}{\#\Hom_{\mathbb{F}_p}
\big(\mathbb{F}_p^j,
(R^{*}/R^{*p})^{-}\big)
}
.$$
\end{conjecture}
More generally, 
as a cruder result, one derives a conjectural formula for the joint distribution of the $p$-rank of $\cl(K)$ and of $\cl(K,c)$, as follows. Fix $j_1,j_2$ two non-negative integers.
\begin{conjecture}
\label{c:wc}
As $K$ varies among imaginary quadratic number fields of type $R$, the proportion of them for which $\dim_{\mathbb{F}_p}(\cl(K)[p])=j_1$ and $\dim_{\mathbb{F}_p}(\cl(K,c)[p])=j_2$ approaches
$$ \mu_{\emph{CL}}(G \in \mathcal{G}_p: \dim_{\mathbb{F}_p}(G[p])=j_1) 
\frac
{\#\{\phi:\mathbb{F}_p^{j_1} \to (
R^{*}/R^{*p}
)^{-}
: \rk(\phi)=\rk_p(R^{*})-(j_2-j_1))\}}
{\#\Hom_{\mathbb{F}_p}(\mathbb{F}_p^{j_1},(
R^{*}/R^{*p}
)^{-})}.$$
\end{conjecture} 
The statements analogous to Conjectures~\ref{conj:3} and \ref{c:wc} for $p=2$ will be proved in Theorem ~\ref{theorem: joint distribution of 4-ranks}, with a more explicit version provided by Theorem ~\ref{theorem: joint distribution of 4-ranks, explicit version}.
\subsection{Agreement with Varma's results}
\label{s:var}
In this section we 
make a certain choice for $f$ in Heuristic assumption~\ref{h:lebkov}
with the aim of stating conjectures for the average of $p$-torsion of ray class groups.
These statements were  previously
proved for $p=3$ 
by Varma~\cite{arXiv:1609.02292}.
In fact, the present paper partly began as an effort to fit her results into a general heuristic framework.

For an element $S \in \mathcal{S}_p(R)$, denote by $M(S)$ the isomorphism class of the middle term of the sequence corresponding to $S$. 
Similarly, for $\theta \in \ext_{\mathbb{Z}_p[C_2]}$
we denote by $M(\theta)$ the isomorphism class of the middle term of the 
equivalence class of
sequences corresponding to $\theta$. 
We will adopt the standard notation $\widehat{A}$
for the dual of a finite abelian group $A$.
\begin{proposition}
\label{prop:just}
We have
$$ \sum_{S \in \mathcal{S}_p(R)}
\#M(S)[p]
\mseq(S)
=
\#{\bigg(\frac{R^{*}}{R^{*p}}\bigg)}^{\!+}
\bigg(1+\#{\Big(\frac{R^{*}}{R^{*p}}}\Big)^{\!-} \bigg).$$
\end{proposition}
\begin{proof}
By the definition of $\mseq$ 
we obtain equality of the sum in our proposition with 
$$
\sum_{G \in \mathcal{G}_p} 
\frac
{\mcl(G)}
{\#\ext_{\mathbb{Z_p}[C_2]}(G,R^{*}[p^{\infty}])} 
{\sum_{\theta \in \ext_{\mathbb{Z}_p[C_2]}(G,R^{*}[p^{\infty}])}\#M(\theta)[p]}
.$$
Again by Proposition \ref{induces surj map} we know that the map $\theta \to \delta_p(\theta)$ is a surjective homomorphism 
$$ \ext_{\mathbb{Z}_p[C_2]}(G,R^{*}[p^{\infty}]) \to 
\Hom_{\mathbb{Z}_p}(G[p],(R^{*}/R^{*p})^{-})
$$
Thus we can rewrite the last sum as
\begin{equation}
\label{eq:sumdelta}
\sum_{G \in \mathcal{G}_p} 
\frac
{\mcl(G)}
{\#
\Hom_{\mathbb{Z}_p}(G[p],(R^{*}/R^{*p})^{-})
} 
\Osum_{\delta} 
\hspace{-0,1cm}
\#R^{*}[p] 
\frac{\#G[p]}{\#\Im(\delta)}
,\end{equation} where  the sum $\Osum$ is taken over 
$\delta$ in $\Hom_{\mathbb{Z}_p}(G[p],(R^{*}/R^{*p})^{-})$.
For each $\chi$ in the dual of  
$(R^{*}/R^{*p})^{-}$
denote by $\mathbf{1}_{\chi}$ 
the indicator function 
of those $\delta$
for which $\chi$ vanishes on the image of $\delta$. 
This allows us
to recast~\eqref{eq:sumdelta}
in the following manner, 
$$ \sum_{G \in \mathcal{G}_p} 
\frac
{\mcl(G)}
{\#\Hom_{\mathbb{Z}_p}(G[p],
(
R^{*}/R^{*p}
)^{-})}
{
\Osum_{\!\!\!\!\!\!\!\delta} 
\hspace{-0,1cm}
\#(R^{*}/R^{*p})^{+}
\#G[p] 
\sum_{\chi \in \widehat{
(
R^{*}/R^{*p}
)^{-}
} 
}
\mathbf{1}_{\chi}(\delta)}
,$$
where $\delta$ varies over all elements in $\Hom_{\mathbb{Z}_p}(G[p],(\frac{R^*}{R^{*p}})^{-})$.
Exchanging the order of summation yields
$$ \sum_{G \in \mathcal{G}_p}
\#(R^{*}/R^{*p})^{+}
\#G[p] \mcl(G)
\sum_{\chi \in \widehat{(\frac{R^{*}}{R^{*p}})^{-}}} 
\frac{\sum_{\delta \in \Hom_{\mathbb{Z}_p}(G[p],(\frac{R^{*}}{R^{*p}})^{-})}\mathbf{1}_{\chi}(\delta)}{\#\Hom_{\mathbb{Z}_p}(G[p],(\frac{R^{*}}{R^{*p}})^{-})}.
$$
The $\chi$-th summand in the last expression equals
 $1$ if $\chi$ is the trivial character
and equals
 $\frac{1}{\#G[p]}$ 
otherwise, thus obtaining
$$ \sum_{G \in \mathcal{G}_p}
\#(
R^{*}/R^{*p}
)^{+}
\#G[p] 
\Big(1+ \frac{\#(
R^{*}/R^{*p}
)^{-}-1}{\#G[p]}\Big) 
\mcl(G)
.$$
Recalling 
the classical equality 
$\sum_{G \in \mathcal{G}_p}\#G[p]
\mcl(G)=2$
provides us with 
$$ 
\#(
R^{*}/(R^{*p})
)^{+} 
\big(2+\#(R^{*}/R^{*p})^{-}-1\big)
=
\#{(
R^{*}/R^{*p}
)}^{+} 
\Big(1+\#{\Big(\frac{R^{*}}{R^{*p}}\Big)}^{-}\Big)
,$$
which concludes our proof. 
\end{proof}

Combining 
Proposition~\ref{prop:just} 
and
Heuristic Assumption~\ref{h:lebkov}
offers
the following.
\begin{conjecture}
\label{conj:generil}
The average value of $\#\cl(K,c)[p]$, as $K$
ranges
among
imaginary quadratic
number fields 
of type $R$ 
ordered by their discriminant, 
is given by  
$$\#\Big(\frac{R^{*}}{R^{*p}}\Big)^{+} \Big(1+\#\Big(\frac{R^{*}}{R^{*p}}\Big)^{-}\Big).$$
\end{conjecture}
In particular we can now derive conjectural formulas for the average size of $\cl(K,c)[p]$ with $K$ varying in
larger
families.

We next 
consider here two cases: in \S\ref{s:unrami}
the case when all the primes dividing $c$ are required to be unramified in $K$, and in \S \ref{s:collall} the case where $K$ 
ranges through all discriminants. 
The letter $l$ 
will refer to a prime until the end of \S\ref{s:2}.

\subsubsection{Collecting unramified discriminants}
\label{s:unrami}
Observe that if $R$ correspond to a splitting type where all the primes dividing $c$ are unramified in $K$, and if $p^2$ does 
not divide $c$ (so there is no contribution to the $p$-part from $p$ itself in case it divides $c$) then we have that
$$
\#\Big(\frac{R^{*}}{R^{*p}}\Big)^{\!+} 
\Big(1+\#\Big(\frac{R^{*}}{R^{*p}}\Big)^{\!-}\Big)=
p^{\#\{l \text{ prime}: \  l\mid c, \ l \equiv  1  \md{p}
\}}
(1+p^{\omega_R(c)})
,$$
where $\omega_R(c)$ is defined by 
\[
\#\{
l \text{ prime}:
l|c,
(l \equiv 1 \ \md{p} \ \text{and } l \text{ is split in } R) \text{  or } 
(l \equiv -1  \md{p} \ \text{and } l \text{ is inert in } R)
\}
.\]
Therefore when we average over all $2^{\omega(c)}$ choices of $R$, using the binomial formula we get
$$
p^{\#\{l \text{ prime}: \ l|c, l \equiv  1 \md{p}\}} \Big(1+\Big(\frac{p+1}{2}\Big)^{\!\#\{l \text{ prime}: \ l|c, l  \equiv  1  \text{ or } -1 \md{p}\}}\Big)
$$
as average value of the size of $\cl(K,c)[p]$ 
when $K$ ranges over
imaginary quadratic
number fields
unramified at all primes dividing $c$, as long as $p^2\nmid c$. 
Instead, if $p^2\mid c$ 
there is an additional contribution from the principal units 
modulo $p^2$ to $\#{(\frac{R^{*}}{R^{*p}})}^{+}   (1+\#(\frac{R^{*}}{R^{*p}})^{-})$, which gives
$$p^{\#\{l \text{ prime}: \ l|c, l \equiv  1  \md{p}\}+1} 
\Big(1+p \Big(\frac{p+1}{2}\Big)^{\#\{l \text{ prime}: \ l|c, l  \equiv  1 \text{ or }  -1  \md{p}\}}\Big)
.$$
This leads to the Conjecture~\ref{conj:7}
that we stated in the introduction. 
The special case $p=3$ of Conjecture~\ref{conj:7}
was recently proved by Varma~\cite[Th.2.(b)]{arXiv:1609.02292}.  
\begin{theorem} [Varma] 
\label{thm:varvarvar}
The average value of $\#\cl(K,c)[3]$ as $K$ ranges over 
imaginary quadratic number fields with $\gcd(\disc(K),c)=1$ is:\\
(1) $$ 3^{\#\{l \emph{ prime}: \ l|c, l \equiv 1  \md{3}\}} (1+2^{\#\{l \emph{ prime}: \ l|c,\ l \neq 3 \}})$$ if $9$ does not divide $c$. \\
(2) 
$$ 3^{\#\{l \emph{ prime}: \ l|c, l \equiv 1  \md{3}\}+1} (1+3 \cdot 2^{\#\{l \emph{ prime}: \ l|c,\ l \neq 3 \}})$$ 
if $9$ divides $c$.
\end{theorem} 

\subsubsection{Collecting all discriminants}
\label{s:collall}

We now consider the case where $K$ is allowed to ramify at the primes dividing $c$. Now we have to evaluate
$$ \sum_{R}\#\Big(\frac{R^{*}}{R^{*p}}\Big)^{+}
\Big(1+\#\Big(\frac{R^{*}}{R^{*p}}\Big)^{-}\Big)
w(R)
,$$
where $R$ varies between all the possible types of ring at $c$, 
and
\[
w(R):=
\lim_{X\to+\infty}
\frac
{\#\{K \in \c{F}_c(R):|\disc(K)|\leq X\}}
{\#\{K \in \c{F}:|\disc(K)|\leq X\}}
.\]
First observe that if $p^2\nmid c$ then 
$$\#\Big(\frac{R^{*}}{R^{*p}}\Big)^{+}=p^{\#\{l \text{ prime}: \ l|c, l  \equiv  1  \md{p}\}}
,$$
while if $p^2|c$ then
$$\#\Big(\frac{R^{*}}{R^{*p}}\Big)^{+}=p^{\#\{l \text{ prime}: \ l|c, l  \equiv  1  \md{p}\}+1}.
$$
Therefore we are left with computing the average of 
$\#\Big(\frac{R^{*}}{R^{*p}}\Big)^{-}$,
over all $R$. But this, as a function of $c$, is multiplicative,
thus we only have to deal with prime powers, i.e. $c=l^n$ for some prime $l$ and some positive integer $n$.
Clearly, the value of this average is $1$ if $l$ is such that
$\gcd(p,l^3-l)=1$. Instead, if $p|l^2-1$ the value of the average is 
$$\frac{1}{l+1}+\frac{(\frac{p+1}{2})
 l}{l+1}=1+\Big(\frac{p-1}{2}\Big) \frac{l}{l+1}
,$$
where the first contribution comes from the $R$ ramified at $l$, and the second from the $R$ unramified at $l$. \footnote{$R$ is said unramified at $l$ if $R/lR$ does not contain non-zero nilpotents. Otherwise $R$ is said ramified at $l$.}
Meanwhile, the value of the average for $p=c$ is
$$\frac{p}{p+1}+\frac{p}{p+1},$$
where the first contribution comes from $R$ ramified at $p$ and the second from
$R$ unramified at $p$. 
Lastly, 
we consider the case  $p^2|c$. Remarkably enough, one observes that the case $p=3$ 
acquires a special status in the computation of this average: indeed $\frac{1}{8}$ of the imaginary quadratics locally at $3$ give the extension 
$\mathbb{Q}_3(\zeta_3)/\mathbb{Q}_3$, and the result for them will be different than for the $\frac{1}{8}$ totally ramified that locally at $3$ become $\mathbb{Q}_3(\sqrt{3})$. Clearly for all $p>3$ there is no $p$-th root of unity in a quadratic extension of $\mathbb{Q}_p$, so, as we will see, in that case the contribution from the two $R$ ramified at $p$ will be the same.

Assume $p=3$. The contribution from powers of $3$ starting from $9$ is
$$\frac{9}{8}+\frac{3}{8}+\frac{9}{4}=\frac{15}{4},$$ 
where the first contribution is from $\mathbb{Q}_3(\zeta_3)$, the second from $\mathbb{Q}_3(\sqrt{3})$ and the third from unramified $R$.
This gives a prediction 
that was previously verified by Varma~\cite[Th.1.(b)]{arXiv:1609.02292}.  
\begin{theorem}
[Varma] 
The average value of $\#\cl(K,c)[3]$ as $K$ ranges through imaginary quadratic number fields ordered by their discriminant 
is: \\
(1) 
$$
3^{\#\{l \emph{ prime}: \ l|c, l \equiv  1 \md{3}\}} 
\Big(
1+\prod_{l|c}\Big(1+ \frac{l}{l+1}\Big)
\Big)
$$ 
if $3$ does not divide $c$, \\
(2)
$$3^{\#\{l \emph{ prime}: \ l|c, l \equiv  1 \md{3}\}} 
\Big(1+\frac{6}{7}  \prod_{l|c}\Big(1+\frac{l}{l+1}\Big)\Big)
$$ 
if $3$ divides $c$ but $9$ does not divide $c$, \\
(3)
$$
3^{\#\{l \emph{ prime}: \ l|c, l \equiv  1 \md{3}\}+1} 
\Big(1+\frac{15}{7}  \prod_{l|c}\Big(1+\frac{l}{l+1}\Big)\Big) 
$$
if $9$ divides $c$.
\end{theorem}

Now assume that $p>3$. Then we get
$$\frac{p}{p+1}+\frac{p^2}{p+1},$$
where the first contribution is from the $R$ ramified at $p$ 
and the second from $R$ unramified at $p$.
Collecting everything together we get the following
prediction.
\begin{conjecture}
Suppose $p>3$. Then the average value of $\#\cl(K,c)[p]$ as $K$ ranges over
imaginary quadratic number fields ordered by their discriminant is: 
\\
(1) 
$$
p^{\#\{l \emph{ prime}: \ l|c, l \equiv  1 \md{p}\}} 
\Big(1+ \prod_{l|c,p|l^2-1}
\Big(1+\frac{p-1}{2}
\frac{l}{l+1}\Big)\Big)  
$$ 
if $p$ does not divide $c$, \\
(2)
$$
p^{\#\{l \emph{ prime}: \ l|c, l \equiv  1 \md{p}\}} 
\Big(
1+\Big(\frac{2p}{p+1}\Big) \prod_{l|c, p|l^2-1}\Big(1+\frac{p-1}{2} \frac{l}{l+1}\Big)\Big)
$$
if $p$ divides $c$ but $p^2$ does not divide $c$, \\
(3) 
$$
p^{\#\{l \emph{ prime}: \ l|c, l \equiv  1 \md{p}\}} 
\Big(1+\Big(\frac{p+p^2}{p+1}\Big) \prod_{l|c, p|l^2-1}\Big(1+ \frac{p-1}{2} \frac{l}{l+1}\Big)\Big)
$$
if $p^2$ divides $c$.
\end{conjecture} 
It would be desirable
to extend Varma's arguments to
prove Conjecture~\ref{conj:generil} for $p=3$.
In particular, it would be informative to see how the proof distinguishes 
between 
the 
cases
$R/3^m=\mathcal{O}_{\mathbb{Q}_3(\zeta_3)}/3^m$ and $R/3^m=\mathcal{O}_{\mathbb{Q}_3(\sqrt{3})}/3^m$, for $m \geq 2$. 
\section{Heuristic and conjectures for $p=2$} \label{section:heuristic at 2}
Let $c$ be an odd positive integer. In this section 
we explain
a heuristic model for the $2$-part of ray class sequences of conductor $c$, in the case that no primes dividing $c$ ramify in the fields. The additional difficulty with respect to the case of $p$ odd, is that $\cl(K)[2^{\infty}]$ does not behave like a random $2$-group (in the sense of Cohen and Lenstra), but instead (as conjectured by Gerth \cite{MR887792}), $2\cl(K)[2^{\infty}]$ is believed to behave like a random $2$-group: the behavior of $\cl(K)[2]$ is governed instead by genus theory which trivially excludes any Cohen--Lenstra behavior for $\cl(K)[2^{\infty}]$, when $K$ varies among usual families of imaginary quadratic number fields.

Our approach will be as follows: we will see that for `most' discriminants of type $R$, $2\cl(K,c)$ is an extension of $2\cl(K)$ with a certain subgroup of $\frac{R^{*}}{\langle -1 \rangle}$, which we will call $W_{R}$. 
Nevertheless, one cannot completely ignore the presence of the class group, since it leaves an additional restriction on such extensions. Namely it forces them to belong to a certain subgroup of the $\ext$, that we will call $\widetilde{\ext}$. From there we will proceed in analogy with the previous section replacing $\ext$ with $\widetilde{\ext}$. Using this heuristic we will offer several predictions which are proved in the subsequent sections.

Since we will only
consider the case that no primes dividing $c$ ramify in the imaginary number fields $K$,
and since we assume that $c$ is odd, we do not lose generality in assuming that $c$ is also square-free: indeed, in our setting, the $2$-part of 
$(\mathcal{O}_K/c)^{*}
/
\langle -1 \rangle$ is no different from the one of 
$(\mathcal{O}_K/c')^{*}/
\langle -1 \rangle,$ 
where $c'$ is the square-free part of $c$. Therefore the choice of a ring type at $c$ amounts
 to the choice of a partition of the set $S_c:=\{l \ \text{prime}: \ l|c \}$ in the disjoint union of two sets $S_c(\text{inert})$ and $S_c(\text{split})$. Then one takes $R:=(\prod_{l \in S_c(\text{inert})} \mathbb{F}_{l^{2}}) \times (\prod_{l' \in S_c(\text{split})} (\mathbb{F}_{l'})^2)$. For such an $R$, the $C_2$-action is given by $l$-Frobenius on the non-split components, and by swapping on the split components. We will call such $R$, unramified at $c$. By a small abuse of notation, we denote by $\mathbb{Z}/c\mathbb{Z}$ the natural image of $\mathbb{Z}/c\mathbb{Z}$ in $R$.

For $R$ unramified at $c$, we define  
\beq{def:wr}
{
W_{R}:= \frac{(\mathbb{Z}/c\mathbb{Z})^{*}}{\langle -1 \rangle}  \l(\frac{R^{*}}{\langle -1 \rangle}\r)^{2} 
\subseteq \frac{R^{*}}{\langle -1 \rangle}
.}
Now fix
some $R$ unramified at $c$. For the remainder of this section we will assume, for simplicity, the imaginary quadratic number field $K$ to have an odd discriminant. 
We shall prove  
that 
one has an exact sequence
\begin{equation}
\label{eq:seq}
2S(K):= (0 \to W_{R} \to 2\cl(K,c) \to 2\cl(K) \to 0)
,\end{equation}
for all imaginary quadratic number fields of type $R$
with the exception of $O(x (\log x)^{-1/\phi(c)})$ 
discriminants up to $x$. 
Indeed, by the theory of ambiguous ideals, one has that 
$$\frac{(\mathcal{O}_{K}/c)^{*}}{\langle -1 \rangle } \cap 2\cl(K,c)=
\langle \{q \ \text{prime and } q|\disc(K)\}\rangle \Big(\frac{(\mathcal{O}_{K}/c)^{*}}{\langle -1 \rangle }\Big)^2.$$
Therefore
it is enough to show that 
the set of positive square-free
$D\leq x$ such that 
\[
\{q \md{c}: q \text{ prime and } q|D\}
\neq 
(\mathbb{Z}/c\mathbb{Z})^{*}
\]
is $O(x (\log x)^{-1/\phi(c)})$.
This cardinality 
is 
\[
\leq 
\sum_{a \in (\mathbb{Z}/c\mathbb{Z})^{*}}
\sum_{\substack{1\leq D \leq X \\p\mid D \Rightarrow p\neq a \md{c}}} 
\hspace{-0,5cm}
\mu(D)^2
\ll \frac{x}{(\log x)^{1/\phi(c)}}
,\]
where the last bound is easily derived
by using~\cite[Eq.(1.85)]{iwko}
with $f$ being the characteristic function of integers all of whose prime 
divisors are not $a\md{c}$.
Identifying $\mathcal{O}_K/c$ with $R$ via a ring isomorphism
gives an identification between $W_R$ and  
\[
\frac{(\mathbb{Z}/c\mathbb{Z})^{*}}{\langle -1 \rangle}
\Big(\frac{(\mathcal{O}_{K}/c)^{*}}{\langle -1 \rangle }\Big)^2
.\]
\begin{definition}
Among the imaginary quadratic number fields of type $R$, we call \emph{strongly} of type $R$, those satisfying 
$$ \frac{(\mathcal{O}_{K}/c)^{*}}{\langle -1 \rangle } \cap 2\cl(K,c)=
\frac{(\mathbb{Z}/c\mathbb{Z})^{*}}{\langle -1 \rangle}
\Big(\frac{(\mathcal{O}_{K}/c)^{*}}{\langle -1 \rangle }\Big)^2
.$$
\end{definition}
Let 
$E(x)$ denote the cardinality 
of negative
discriminants
$1\md{4}$ of absolute value at most 
$x$ and which are of type $R$ but not strongly of type $R$. 
The analysis above can be summarised by the bound 
\beq{eq:acacac}
{
E(x)\ll \frac{x}{(\log x)^{1/\phi(c)}}
.}
One could be tempted to think of the sequence $S_2(K):=2S(K)[2^{\infty}]$ as a `random' sequence, just as in the previous section. This would be incorrect, since the way the sequences $S_2(K)$ are produced naturally puts on them an additional restriction. Namely one has a commutative diagram of $\mathbb{Z}[C_2]$-modules:
$$ \begin{array}{ccc}  0 &\overset{}\to& \frac{(\mathcal{O}_K/c)^{*}}{\langle -1 \rangle} \to  \\  && \uparrow i_1\\ 0 &\underset{}\to& \frac{(\mathbb{Z}/c\mathbb{Z})^{*}}{\langle -1 \rangle}(\frac{(\mathcal{O}_{K}/c)^{*}}{\langle -1 \rangle })^2  \to  \end{array} \begin{array}{ccc}  \cl(K,c) &\overset{\pi}\to& \cl(K) \to 0 \\ \uparrow i_2 && \uparrow i_3\\ 2\cl(K,c) &\underset{}\to& 2\cl(K) \to 0 \end{array}
$$
where $i_1,i_2,i_3$ are the natural inclusion maps, so $i_2$ and $i_3$ consist of isomorphisms between the source groups and the double of the target groups. The top sequence has two obvious properties that are automatically satisfied: 
$$ \pi(\cl(K,c)[2^{\infty}]^{-})=\cl(K)[2^{\infty}] \ \text{and} \
\pi(\cl(K,c)[2^{\infty}]^{+})=\cl(K)[2]
.$$
The first property is equivalent 
to
the sequence remaining
exact after taking $(1+\tau)$-torsion, where $\tau$ is the generator of $C_2$. Indeed, this is equivalent to 
the natural map $$\cl(K)[2^{\infty}] \to \frac{\frac{R^{*}}{\langle -1 \rangle}}{ (\tau +1)\frac{R^{*}}{\langle -1 \rangle}}$$ 
being
the $0$-map, which holds since the norm of an integral ideal is always an integer.
The second property follows from the fact that we are looking at families of discriminants coprime to $c$. Therefore we are allowed to lift a prime ideal $\mathfrak{q}$ lying above a prime $q$ dividing $\disc(K)$, using the class of the ideal $\mathfrak{q}$ in $\cl(K,c)$: this class will still be a fixed point, since it is the class of a $\tau$-invariant \emph{ideal}. This motivates the following:
\begin{definition}
Let $G$ be a finite abelian $2$-group, viewed as a $C_2$ module with the $-\text{id}$-action. We say that an element $\theta$ of $\text{Ext}_{\mathbb{Z}_2[C_2]}(G,W_R[2^{\infty}])$:
$$ \theta: 1 \to W_R[2^{\infty}] \to B \to G \to 1
$$
is \emph{embeddable} if there is an exact sequence of $\mathbb{Z}_2[C_2]$-modules 
$$ 1 \to \frac{R^{*}}{\langle -1 \rangle}[2^{\infty}] \to \tilde{B} \to \tilde{G} \to 1 $$
and a commutative diagram of $\mathbb{Z}_2[C_2]$-modules
$$ \begin{array}{ccc}  0 &\overset{}\to& \frac{(R)^{*}}{\langle -1 \rangle}[2^{\infty}] \to  \\  && \uparrow i_1\\ 0 &\underset{}\to& W_R[2^{\infty}]  \to  \end{array} \begin{array}{ccc}  \tilde{B} &\overset{\pi}\to& \tilde{G} \to 0 \\ \uparrow i_2 && \uparrow i_3\\ B &\underset{}\to& G \to 0 \end{array}
$$
where: \\
$\bullet$ The map $\pi:\tilde{B} \to \tilde{G} \to 1 $ satisfies
$$ \pi(\tilde{B}^{-})=\tilde{G} \ \text{and} \ \pi(\tilde{B}^{+})=\tilde{G}[2].
$$ \\
$\bullet$ The maps $i_2$ and $i_3$ are isomorphisms between the source groups and the double of the target groups. The map $i_1$ is the natural inclusion. 
\end{definition}
We denote the set of embeddable extensions by $\widetilde{\ext}_{\mathbb{Z}_2[C_2]}(G,W_R[2^{\infty}])$. 
It will be clear by Proposition~\ref{delta maps for ext tilde}, that the two following
sets 
do not always coincide:
\[
\widetilde{\ext}_{\mathbb{Z}_2[C_2]}(G,W_R[2^{\infty}]), 
\ext_{\mathbb{Z}_2[C_2]}(G,W_R[2^{\infty}])
.\]
On the other hand, the set of embeddable extensions has the algebraic structure that allows us to proceed in perfect parallel with the previous section. 
\begin{proposition}
One has that $\widetilde{\ext}_{\mathbb{Z}_2[C_2]}(G,W_R[2^{\infty}])$ is a subgroup of  $\ext_{\mathbb{Z}_2[C_2]}(G,W_R[2^{\infty}])$ stable under the action of $\aut_{\emph{ring}}(R) \times \aut_{\emph{ab.gr.}}(G)$.
\begin{proof}
Let 
$$ \begin{array}{ccc}  0 &\overset{}\to& \frac{(R)^{*}}{\langle -1 \rangle}[2^{\infty}] \to  \\  && \uparrow i_1\\ 0 &\underset{}\to& W_R[2^{\infty}]  \to  \end{array} \begin{array}{ccc}  \tilde{B} &\overset{\pi}\to& \tilde{G} \to 0 \\ \uparrow i_2 && \uparrow i_3\\ B &\underset{f}\to& G \to 0 \end{array}
$$
and
$$ \begin{array}{ccc}  0 &\overset{}\to& \frac{(R)^{*}}{\langle -1 \rangle}[2^{\infty}] \to  \\  && \uparrow i_1\\ 0 &\underset{}\to& W_R[2^{\infty}]  \to  \end{array} \begin{array}{ccc}  \tilde{B'} &\overset{\pi'}\to& \tilde{G'} \to 0 \\ \uparrow i_2' && \uparrow i_3'\\ B' &\underset{f'}\to& G \to 0 \end{array}
$$
be
two embeddable extensions equipped with their respective diagrams. We now consider the following commutative diagram of $\mathbb{Z}_2[C_2]$-modules,
$$ \begin{array}{ccc}  0 &\overset{}\to& \frac{(R)^{*}}{\langle -1 \rangle}[2^{\infty}] \to  \\  && \uparrow i_1\\ 0 &\underset{}\to& W_R[2^{\infty}]  \to  \end{array} \begin{array}{ccc}  (\tilde{B} \times_{G}\tilde{B'})/Y' &\overset{\pi \times \pi'}\to& \tilde{G} \times_{G} \tilde{G'} \to 0 \\ \uparrow i_2 \times i_2' && \uparrow i_3 \times i_3'\\  (B \times_{G} B')/Y &\underset{f \times f'}\to& G \to 0 \end{array}
$$
where $\tilde{B} \times_{G} \tilde{B'}:=\{(b_1,b_2) \in \tilde{B} \times \tilde{B'}: 2\pi(b_1)=2\pi'(b_2)\}$, while $Y'$ denotes the antidiagonal embedding of $\frac{(R)^{*}}{\langle -1 \rangle}[2^{\infty}]$ in $\tilde{B} \times_{G} \tilde{B'}$. Similarly $ B \times_{G} B':=\{(b_1,b_2) \in B \times B': f(g_1)=f'(g_2)\}$, with $Y$ denoting the anti-diagonal embedding of $W_R[2^{\infty}]$, and 
$$ \tilde{G} \times_{G} \tilde{G'}:=\{(g_1,g_2) \in \tilde{G} \times \tilde{G'}: 2g_1=2g_2\}.$$ 
There is an obviously induced compatible $C_2$ action on each terms and one can deduce that 
$$ (\pi \times \pi')(((\tilde{B} \times_{G}\tilde{B'})/Y')^{-})=\tilde{G} \times_{G} \tilde{G'}   
\text{ and }
(\pi \times \pi')(((\tilde{B} \times_{G}\tilde{B'})/Y')^{+})=(\tilde{G} \times_{G} \tilde{G'})[2]
$$
using the fact that individually $\pi$ and $\pi'$ satisfy the respective property.

On the other hand, by construction one has that $i_2 \times i_2'$ and $i_3 \times i_3'$ are isomorphisms between the source groups and the double of the targets. This shows that $\widetilde{\ext}_{\mathbb{Z}_2[C_2]}(G,W_R[2^{\infty}])$ is closed under addition because the sequence
$0 \to W_R[2^{\infty}] \to (B \times_G B')/Y \to G \to 0$
represents the class of the Baer sum of the two embeddable sequences in $\ext_{\mathbb{Z}_2[C_2]}(G,W_R[2^{\infty}])$. Since $\ext_{\mathbb{Z}_2[C_2]}(G,W_R[2^{\infty}])$ is finite, in order to conclude that $\widetilde{\ext}_{\mathbb{Z}_2[C_2]}(G,W_R[2^{\infty}])$ is a subgroup, one is only left to show that $\widetilde{\ext}_{\mathbb{Z}_2[C_2]}(G,W_R[2^{\infty}])$ is non-empty. To this end we refer the reader to Proposition \ref{delta maps for ext tilde}, which in particular implies that $\widetilde{\ext}_{\mathbb{Z}_2[C_2]}(G,W_R[2^{\infty}])$ is non-empty (alternatively one could also directly prove that the split sequence is embeddable, which one can indeed show using the same steps of the proof of Proposition \ref{delta maps for ext tilde}).
Finally, given an embeddable sequence
$$\begin{array}{ccc}  0 &\overset{}\to& \frac{(R)^{*}}{\langle -1 \rangle}[2^{\infty}] \overset{g}\to  \\  && \uparrow i_1\\ 0 &\underset{}\to& W_R[2^{\infty}] \overset{h} \to  \end{array} \begin{array}{ccc}  \tilde{B} &\overset{\pi}\to& \tilde{G} \to 0 \\ \uparrow i_2 && \uparrow i_3\\ B &\underset{f}\to& G \to 0 \end{array}
$$
and a pair $(\phi_1,\phi_2) \in \aut_{\text{ring}}(R) \times \aut_{\text{ab.gr.}}(G)$, we can consider
$$\begin{array}{ccc}  0 &\overset{}\to& \frac{(R)^{*}}{\langle -1 \rangle}[2^{\infty}] \overset{g \phi_1}\to  \\  && \uparrow i_1\\ 0 &\underset{}\to& W_R[2^{\infty}] \overset{h\phi_1} \to  \end{array} \begin{array}{ccc}  \tilde{B} &\overset{\pi}\to& \tilde{G} \to 0 \\ \uparrow i_2 && \uparrow i_3\phi_2^{-1}\\ B &\underset{\phi_2f}\to& G \to 0 \end{array}
$$
which
gives an embeddability diagram for the sequence
$$(\phi_1,\phi_2)(0 \to W_R[2^{\infty}] \to B \to G \to 0)
$$
showing that $\widetilde{\ext}_{\mathbb{Z}_2[C_2]}(G,W_R[2^{\infty}])$ is stable under the action of $\aut_{\text{ring}}(R) \times \aut_{\text{ab.gr.}}(G)$.
\end{proof}
\end{proposition}
Denote by $\mathcal{G}_2$ a set of representatives of isomorphism classes of finite abelian $2$-groups, viewed as $C_2$-modules under the action of $-\Id$. For an imaginary quadratic number field $K$, denote by $G_2(K)$ the unique representative of $2\cl(K)[2^{\infty}]$ in $\mathcal{G}_2$. Suppose $K$ is strongly of type $R$. Then 
$(\mathcal{O}_K/c)^{*}/\langle -1 \rangle$ can be identified with $R^{*}/\langle -1 \rangle$ via any restriction of a ring isomorphism, that is via any element of 
$\isom_{\text{ring}}(\mathcal{O}_K/c,R)$. 
Furthemore, 
we can identify $2\cl(K)[2^{\infty}]$ and $G_2(K)$  via any element of 
$\isom_{\text{ab.gr.}}(\cl(K)[2^{\infty}],G)$.
Therefore applying $\isom_{\text{ring}}(\mathcal{O}_K/c,R) \times \isom_{\text{ab.gr.}}(2\cl(K)[2^{\infty}],G_2(K))$ to 
$S_2(K)$, we obtain a unique orbit 
$$ O_{c,2}(K) \in 
\widetilde{\ext}_{\mathbb{Z}_2[C_2]}(G_2(K),W_R[2^{\infty}])/(\aut_{\text{ring}}(R) \times \aut_{\text{ab.gr.}}(G)).$$
For $K$ strongly of type $R$ we use the notation
$$S'_2(K):=(G_2(K),O_{c,2}(K)).
$$
If $K$ is not strongly of type $R$, we set $S'_2(K)$ to be the symbol $\bullet$.
We now proceed by
offering a heuristic model for $S'_2(K)
$ as $K$ varies among imaginary quadratic number fields of type $R$. 
Let $R$ be an unramified ring at $c$ 
and denote by $\mathcal{G}_2$ a set of representatives of isomorphism classes of finite abelian $2$-groups, viewed as $C_2$-modules under the action of $-\Id$. Denote by $\mathcal{S}_2(R)$ the union of the singleton $\{\bullet\}$ and of the set of equivalence classes of pairs $(G,\theta)$, where $G \in \mathcal{G}_{2}$, $\theta \in \widetilde{\ext}_{\mathbb{Z}_2[C_2]}(G,W_R[2^{\infty}])$
and the equivalence is defined as follows:
two pairs $(G_1,\theta_1), (G_2,\theta_2)$ are identified if $G_1=G_2$ and $\theta_1, \theta_2$ are in the same $\aut_{\text{ring}}(R) \times \aut_{\text{ab.gr.}}(G)$-orbit. Denote by $\widetilde{\mathcal{S}_2}(R)$ the union of the singleton $\{\bullet\}$ and the set of pairs $(G,\theta)$, where $G \in \mathcal{G}_{2}$ and $\theta \in \widetilde{\ext}_{\mathbb{Z}_2[C_2]}(G,W_R[2^{\infty}])$,
thus bringing into play
the 
quotient map
$$\pi: \widetilde{\mathcal{S}}_2(R) \to \mathcal{S}_2(R).$$
Consider the sigma algebra generated by all subsets
on $\widetilde{\mathcal{S}}_2(R)$, as well as on $\mathcal{S}_2(R)$,
and equip $\widetilde{\mathcal{S}}_2(R)$ with the measure
$$\widetilde{\mu}_{\text{seq}}((G,\theta))
:=
\frac{\mcl(G)}{\#\widetilde{\ext}_{\mathbb{Z}_2[C_2]}(G,W_R[2^{\infty}])},
\widetilde{\mu}_{\text{seq}}(\{\bullet\})=0
,$$
where $\mcl$ denotes, as usual, the Cohen--Lenstra probability measure on $\mathcal{G}_{2}$ that
 gives to each abelian $2$-group $G$ weight inversely proportional to the size of the automorphism group of $G$. 
Push forward, via $\pi$, the measure $\widetilde{\mu}_{\text{seq}}$ to a measure $\mseq$ on $\mathcal{S}_2(R)$. It is clear by construction that $\widetilde{\mu}_{\text{seq}}$ and $\mseq$ are probability measures.  
\\
\\ 
The heuristic assumption that we propose for the $2$-part of ray class sequences of conductor $c$ of imaginary quadratic fields of type $R$ is as follows.
\begin{heuristic}
For any `reasonable' function $f: \mathcal{S}_2(R) \to \mathbb{R}$ one has that, as $K$ varies among imaginary quadratic number fields of type $R$, the following equality of averages takes place
\begin{align*}
\lim_{X \to \infty}\frac{\sum_{-\disc(K)\leq X}f(S'_2(K))}{\#\{-\disc(K)\leq X\}}
=&\sum_{S \in \mathcal{S}_2(R)}f(S)\mseq(S).
\end{align*}
\end{heuristic}
As a consistency check, observe that the above identity of average takes place if one chooses as $f$ the indicator function of $\{\bullet\}$: indeed, since the number of $K$ with $\disc(K)\leq X$ that are not strongly of type $R$
is at most 
$\ll_{c} X (\log X)^{-1/\phi(c)}$, we see that 
we obtain $0$
in the left side,
while in the right side we obtain $0$ by definition. 
Clearly one can readily formulate the analogues of Conjectures ~\ref{conj:1} and~\ref{con:eibaar}.
We shall instead opt to
devote the rest of the section to the analogues of Conjectures~\ref{conj:3}-\ref{c:wc}.

If $\alpha \in  R^{*}/\langle -1 \rangle $
then $\alpha^2 N(\alpha) \in W_R$, where 
$N(\cdot)$ is the norm-function with respect to the $C_2$-action prescribed to $R^{*}/\langle -1 \rangle$: indeed both $\alpha^2$ and $N(\alpha)$ are in $W_R$. We define the map 
$g_{R}: R^{*}/\langle -1 \rangle \to W_R$
given by
$\alpha \mapsto \alpha^2  N(\alpha)$. With a small abuse of notation, we use the same notation for the induced map $g_{R}:
\frac
{
R^{*}/\langle -1 \rangle
}
{
(R^{*}/\langle -1 \rangle)^2
} \to W_R/2W_R$
and we denote by $\Im(g_R)$ the image of $g_R$ in $W_R/2W_R$. 
\begin{proposition}\label{delta maps for ext tilde}
The image of the natural map
$$\widetilde{\ext}_{\mathbb{Z}_2[C_2]}(G,W_R) \to \Hom_{\mathbb{F}_2[C_2]}(G[2],W_R/2W_R)
$$
is 
$$ \Hom_{\mathbb{F}_2[C_2]}(G[2],\Im(g_R)) \ \ (=\Hom_{\mathbb{F}_2}(G[2],\Im(g_R))).$$
\end{proposition}
\begin{proof}
Consider $\theta$ an embeddable sequence 
$$\begin{array}{ccc}  0 &\overset{}\to& \frac{(R)^{*}}{\langle -1 \rangle}[2^{\infty}] \to  \\  && \uparrow i_1\\ 0 &\underset{}\to& W_R[2^{\infty}]  \to  \end{array} \begin{array}{ccc}  \tilde{B} &\overset{\pi}\to& \tilde{G} \to 0 \\ \uparrow i_2 && \uparrow i_3\\ B &\underset{f}\to& G \to 0 \end{array}
$$
and pick $b \in G[2]$. By definition of embeddability there exist $\mathfrak{b}$ in $\tilde{B}^{+}$ such that $\pi(\mathfrak{b})=i_3(b)$. On the other hand we can find $x \in \tilde{B}$ such that $\pi(2x)=i_3(b)$. Therefore there 
exists an element
$\alpha \in  \frac{(R)^{*}}{\langle -1 \rangle}[2^{\infty}]$ such that $\mathfrak{b} 
{\alpha}^{-1}=x^2$,
which implies that
$\mathfrak{b}^2 N(\alpha)^{-1}=N(x)^2$. Furthermore,
$2x$ is in $B$,
hence
we have that $\delta_2(\theta)(b)=\mathfrak{b}^2
\alpha^{-2}$ as an element of $W_R/2W_R$. However note that
 $N(x)^2 \in 2W_R$: indeed, by definition of embeddability, we can always write $x=x^{-}\beta$ with $x^{-}$ an anti-fixed point and $\beta \in \frac{R^{*}}{\langle -1 \rangle}$, so that $N(x)^2=N(\beta)^2 \in W_R$. Therefore we find that $\delta_2(\theta)(b)=N(\alpha)\alpha^2$, i.e. $\delta_2(\theta)(b) \in \Im(g_R)$.

Conversely, we prove that given a $C_2$-map $\delta_0: G[2] \to \Im(g_R)$, there exists a $\theta \in \widetilde{\ext}_{\mathbb{Z}_2[C_2]}$ such that $\delta_2(\theta)=\delta_0$.
Firstly observe that 
$ \Hom_{\mathbb{F}_2[C_2]}(G[2],\Im(g_R))=\Hom_{\mathbb{F}_2}(G[2],\Im(g_R)))$, since $\tau$ clearly fixes $N(\alpha)$ for any $\alpha$ in $R$
and $\alpha^2\tau(\alpha^2)=N(\alpha)^2 \in 2W_R$, therefore $\tau$ acts trivially on $\Im(g_R)$ (see Lemma ~\ref{trivial action} for a more general fact). 
Thus pick $\delta_0 \in \Hom_{\mathbb{F}_2}(G[2],\Im(g_R)))$. We divide the construction of $\theta$ and its embedding in four steps: \\
\emph{Step 1:} Observe that $\alpha^2N(\alpha)=\frac{\alpha^2}{N(\alpha)}N(\alpha)^2=\frac{\alpha}{\tau(\alpha)}N(\alpha)^2$. Since $N(\alpha)^2 \in 2W_R[2^{\infty}]$, we conclude that any element of $\Im(g_R)$ can be represented as $\frac{\alpha}{\tau(\alpha)}$ for some $\alpha \in \frac{R^{*}}{\langle -1 \rangle}[2^{\infty}]$. \\
\emph{Step 2:} Write $G=\langle e_1 \rangle \oplus \ldots \oplus \langle e_j \rangle$, with the order of $e_i$ being $2^{m_i}$ for a positive integer $m_i$, for each $i \in \{1,\ldots,j\}$. Therefore $G[2]=\langle 2^{m_1-1}e_1 \rangle \oplus \ldots \oplus \langle 2^{m_j-1}e_j\rangle$
and now,
use Step 1
for each $i \in \{1,\ldots,j\}$
to construct $\alpha_i \in \frac{R^{*}}{\langle -1 \rangle}[2^{\infty}]$
such that $\delta_0(2^{m_i-1}e_i)=\frac{\alpha_i}{\tau(\alpha_i)}$. \\
\emph{Step 3:} Embed $G$ in a group $\tilde{G}=\langle \tilde{e}_1 \rangle \oplus \ldots \oplus \langle\tilde{ e}_{j} \rangle \oplus \langle d_1 \rangle \oplus \ldots \oplus \langle d_h \rangle$, with the rules   $2\tilde{e}_i=e_i$ for every $i$ in $\{1,\ldots,j\}$, $2d_s=0$ for every $s \in \{1, \ldots ,h\}$ and $h \geq \rk_{2}((\mathbb{Z}/c\mathbb{Z}^{*}))-1$. Take an extension $\theta \in \ext_{\mathbb{Z}_2}(\tilde{G},\frac{R^{*}}{\langle -1 \rangle})$ such that for every $i \in \{1, \ldots ,j\}$ one has that $\delta_{2^{m_i}}(\theta)(\tilde{e}_i)=\frac{\alpha_i}{\tau(\alpha_i)}$ and such that $\langle \{\delta_2(\theta)(d_1), \ldots ,\delta_2(\theta)(d_{h})\} \rangle=\Im((\mathbb{Z}/c\mathbb{Z})^{*} \to W_R/2W_R)$. Call $\tilde{B}$ the middle term of this extension. Pick 
$\tilde{e}'_{1},\ldots,\tilde{e}'_{j}$ liftings of $e_{1},\ldots,e_{j}$ with the property that 
$2^{m_i}\tilde{e}'_{i}=\frac{\alpha_i}{\tau(\alpha_i)}$ for all $i$ in $\{1,\ldots ,j\}$. 
Choose also ${d}'_1,\ldots,{d}'_h$ liftings
of $d_1,\ldots,d_{h}$ in $\tilde{B}$ and put $2\tilde{B}=B$. 
Observe that by construction the kernel of $B \to G$ is $W_R[2^{\infty}]$. This gives a commutative diagram of $\mathbb{Z}_2[C_2]$-modules,
$$\begin{array}{ccc}  0 &\overset{}\to& \frac{(R)^{*}}{\langle -1 \rangle}[2^{\infty}] \to  \\  && \uparrow i_1\\ 0 &\underset{}\to& W_R[2^{\infty}]  \to  \end{array} \begin{array}{ccc}  \tilde{B} &\overset{\pi}\to& \tilde{G} \to 0 \\ \uparrow i_2 && \uparrow i_3\\ B &\underset{f}\to& G \to 0. \end{array}
$$
\emph{Step 4:} Define $A_1:=\langle \{\tilde{e}'_1,\ldots,\tilde{e}'_j \rangle$,
$A_2:=\langle\{{d}'_1,\ldots,{d}'_h\} \rangle$
and
$A:=\langle A_1,A_2 \rangle$. 
Consider $A_1$ as a $C_2$-module with the
 $-\text{Id}$-action and $A_2$ with the $\text{Id}$-action. Observe that, by construction, the $C_2$-action on $A_1$ and $A_2$ restrict to the same $C_2$-action on $A_1 \cap A_2$. Therefore the $C_2$-action extend to an action on $A$.
Observe that, by construction, the $C_2$-action on $A$ and $\frac{R^{*}}{\langle -1 \rangle}[2^{\infty}]$ restricts to the same $C_2$-action on  $A \cap \frac{R^{*}}{\langle -1 \rangle}[2^{\infty}]$. It is also clear that $\langle A,\frac{R^{*}}{\langle -1 \rangle}[2^{\infty}]\rangle=\tilde{B}$. Therefore one can put on $\tilde{B}$ a $C_2$-action which restricted to $A$ is $-\text{Id}$ and restricted to $\frac{R^{*}}{\langle -1 \rangle}$ is the usual action.  This turns the above diagram into a diagram of $C_2$-modules, and we want to prove that the top sequence remains exact when we take $(1+\tau)$-torsion and when we take $(1-\tau)$-torsion. 
But by construction
\begin{align*}
(1+\tau)(\tilde{B})
&=(1+\tau)\Big(\langle A_1,A_2, R^{*}/ \langle -1 \rangle\rangle\Big)
=
(1+\tau)\Big(\langle A_2, R^{*}/ \langle -1 \rangle \rangle\Big)
\\&=
\langle
2A_2,(1+\tau)(R^{*}/ \langle -1 \rangle) \rangle 
\subseteq 
\langle (1+\tau)(
R^{*}/ \langle -1 \rangle
) \rangle
\end{align*}
and 
\begin{align*}
(1-\tau)(\pi^{-1}(\tilde{G}[2])
&=(1-\tau)(\langle A_1 \cap \ker(2\pi),R^{*}/ \langle -1 \rangle\rangle)
\\&=
\langle 2(A_1 \cap \ker(2\pi)), (1-\tau)(R^{*}/ \langle -1 \rangle) \rangle 
\\&\subseteq 
 (1-\tau)(R^{*}/ \langle -1 \rangle) 
,\end{align*}
where the last two inclusions follow from Step 3.
This shows that the
diagram above is an embedding, 
concluding the proof that $\delta_0$ can be realized as $\delta_2(\theta)$ for some $\theta$ in $\widetilde{\ext}_{\mathbb{Z}_2[C_2]}(G,W_R[2^{\infty}])$ (i.e. $0 \to W_R[2^{\infty}] \to B \overset{f} \to G \to 0 $). 
\end{proof}
If $K$ is strongly of type $R$, we denote by 
$\delta_2(K)$
the map $\delta_2(S_2(K))$. By choosing any ring identification in $\text{Isom}_{\text{ring}}(\mathcal{O}_K/c,R)$ and any identification in $\text{Isom}_{\text{ab.gr.}}(2\cl(K),G_2(K))$ we obtain an $\aut_{\text{ring}}(R)$-orbit of subspaces of $W_R/2W_R$. On the other hand this orbit is composed of a single element due to the following fact:
\begin{lemma}
\label{trivial action}
The action of $\aut_{\emph{ring}}(R)$ on $\Im(g_R)$ is trivial.
\end{lemma}
\begin{proof}
Consider the ring decomposition $R=\prod_{l|c}R/lR$.
It is clear that 
the following holds,
$\aut_{\text{ring}}(R)=\prod_{l|c}\aut_{\text{ring}}(R/lR)$. 
On the other hand, this decomposition is compatible with $g_R$, i.e. $g_R=\prod_{l|c}g_{R/lR}$, where $\prod$ of maps 
is to be thought of as
the map obtained by
applying the maps coordinatewise. 
This reduces the claim to $c=l$ a prime number. In that case one has that $\alpha^2 
 \tau(\alpha)^2=N(\alpha)^2
 $, but $N(\alpha)^2$ is in $2W_R$, therefore, modulo $2W_R$, one has that $\alpha^2N(\alpha)$ is fixed by $\tau$.
\end{proof}
Hence we see that
$\Im(\delta_2(K))$ can be identified with a well-defined subgroup of $\Im(g_R)$. 
We will keep denoting this subgroup as $\Im(\delta_2(K))$.
Moreover, thanks to Proposition~\ref{delta maps for ext tilde} and the fact that the pushforward, via an epimorphism, of the counting probability measure induces the counting probability measure on the target group, we readily obtain the prediction of the distribution of the pair $(\#(2\cl(K))[2],\Im(\delta_2(K)))$.

Fix a subspace $Y \subseteq \Im(g_R)$
and a non-negative integer $j$.
\begin{prediction}
\label{pred:croisant}
As $K$ varies among imaginary quadratic number fields of type $R$, we have the following equality
\begin{align*}
&\lim_{X \to \infty}\frac{\#\{K: -\disc(K)\leq X, \#(2\cl(K))[2]=2^j \ \text{and} \ \Im(\delta_2(K))=Y\}}
{\#\{K: -\disc(K)\leq X \}}
\\
&=\mu_{\emph{CL}}(G \in \mathcal{G}_2:\#G[2]=2^j)
\frac{\#\epi_{\mathbb{F}_2}(\mathbb{F}_2^j,Y)}{\#\Hom_{\mathbb{F}_2}(\mathbb{F}_2^j,\Im(g_R))}.
\end{align*}
\end{prediction}
This will be proved in 
Theorem~\ref{theorem:distribution of delta maps}, but see also 
Theorem~\ref{theorem: joint distribution of 4-ranks, explicit version} for a more explicit statement. \\
A crucial step is to deduce it from a statement about \emph{mixed moments}. Indeed, observe that to know the pair
$$ (\#G[2],\Im(\delta: G[2] \to \Im(g_R)))
$$
is equivalent to knowing for each $\chi$ in the dual group $\widehat{\Im(g_R)}$, the value of 
$$m_{\chi}(\delta):=\#\text{ker}(\chi(\delta)).$$ 
For each $\chi \in \widehat{\Im(g_R)}$, fix a non-negative integer $k_{\chi}$. 
\begin{notation} 
For any  
function 
$\widehat{\Im(g_R)} \to \mathbb{Z}_{\geq 0}$,
$\chi \mapsto k_{\chi}$, 
we will use the notation 
$$ 
\bk
:=\sum_{\chi \in \widehat{\Im(g_R)} }k_{\chi}.$$ 
\end{notation}
Pick a random subset of $\widehat{\Im(g_R)}$ by choosing each character $\chi$ independently at random with the rule that $\chi$ is 
not in the
set
with probability $\frac{1}{2^{k_{\chi}}}$ and 
that
$\chi$ is in the set
with probability $\frac{2^{k_{\chi}}-1}{2^{k_{\chi}}}$.
For a subspace $Y \subseteq \widehat{\Im(g_R)}$
denote by $\mathbb{P}_{(k_{\chi})}(Y)$ the probability that such a random subset generates $Y$. 
Observe that if $\dim(Y)> 
\bk
$ then  $\mathbb{P}_{(k_{\chi})}(Y)=0$: indeed,
in that case we select with probability $1$ less characters than $\dim_{\mathbb{F}_2}(Y)$, so they
they generate $Y$
with 
zero probability. 
Denote by $\mathcal{N}_2(j)$ the number of vector subspaces of $\mathbb{F}_2^j$. If $j<0$, 
we shall
make sense of the expression $0 \cdot \mathcal{N}_2(j)$ by setting it equal to
$0$.

The following proposition reveals
the value predicted by the heuristic model for the $(k_{\chi})_{\chi \in \widehat{\Im(g_R)}}$-mixed moment. 
In what follows
we use the convention
$m_{\chi}(\delta_S)=0$
if we have $S=\bullet \in \mathcal{S}_2(R)$.

\begin{proposition} \label{heur.pr.m.m.}
One has that 
$$\sum_{S \in \mathcal{S}_2(R)}
\mseq(S)
\prod_{\chi \in \widehat{\Im(g_R)}}m_{\chi}(\delta_S)^{k_{\chi}}
=\sum_{Y \subseteq \widehat{\Im(g_R)}} \mathbb{P}_{(k_{\chi})}(Y) \mathcal{N}_{2}(\bk-\dim(Y))
.$$
\end{proposition}
We do not spell out the proof of Proposition~\ref{heur.pr.m.m.}
because it is identical to the proof of
Proposition~\ref{heuristic comp. of multimoments}
which we will provide in \S\ref{section:special divisors}.

Proposition~\ref{heur.pr.m.m.} leads to the following prediction.
\begin{prediction}
\label{pred:croisant2}
As $K$ varies among imaginary quadratic number fields of type $R$, the following equality of averages takes place
$$\lim_{X \to \infty} \frac{\sum_{-\disc(K)\leq X}\prod_{}m_{\chi}(\delta_2(K))^{k_{\chi}}}
{\#\{K:-\disc(K)\leq X\}}
=
\sum_{V \subseteq \widehat{Im(g_R)}}\mathbb{P}_{(k_{\chi})}(V)\mathcal{N}_2(\bk-\dim(V))
.$$
\end{prediction}
A stronger statement will be proved in Theorem~\ref{theorem:multi-moments of ray class seq}.

As a cruder result, one derives a prediction for the joint-distribution of the $4$-ranks of the class group and the ray class group. Let $j_1,j_2$ be two non-negative integers. Then we have the following prediction.
\begin{prediction}
\label{pred:croisant3}  
As $K$ varies among imaginary quadratic number fields of type $R$, we have the following equality
\begin{align*}
&\lim_{X \to \infty} 
\frac{\#\{K: -\disc(K)\leq X, \rk_4(\cl(K))=j_1, \rk_4(\cl(K,c))=j_2\}}{\#\{K:-\disc(K)\leq X \}}
\\
&=
\mu_{\emph{CL}}(G \in \mathcal{G}_2:\dim_{\mathbb{F}_2}(G[2])=j_1)
\frac{\#\{\phi \in \Hom_{\mathbb{F}_2}(\mathbb{F}_2^{j_1},\Im(g_R)): \rk(\phi)=\rk_{2}(W_R)-(j_2-j_1)\}}
{\#\Hom_{\mathbb{F}_2}(\mathbb{F}_2^{j_1},\Im(g_R))}
.
\end{align*}
\end{prediction}
This will be proved in Theorem ~\ref{theorem: joint distribution of 4-ranks}, but see also Theorem ~\ref{theorem: joint distribution of 4-ranks, explicit version} for a more explicit law.
Similarly, the heuristic of the present section can be used to conjecturally predict the distribution of the pair
$
(\rk_{2^m}(\cl(K)),\rk_{2^m}(\cl(K,c)))
$
among imaginary quadratic number fields $K$
with $\gcd(\disc(K),c)=1$. 
For reasons of space we do not 
explicitly state
such a conjecture 
but it is implicitly given in the present section;
such a conjecture might be within reach
given the recent work of Smith~\cite{arXiv:1702.02325v2}.

 \section{Special divisors and $4$-rank} 
\label{section:special divisors}

Let $D$ be a square-free odd positive integer. 
In this section we introduce the notion of 
\textit{special divisors} of $D$,
which will be instrumental in our proof of 
Theorems~\ref{theorem:multi-moments of ray class seq},~\ref{theorem:distribution of delta maps},~\ref{theorem: joint distribution of 4-ranks},
and~\ref{theorem: joint distribution of 4-ranks, explicit version}.  
We call a positive divisor $d$ of $D$ \textit{special} 
if $d$ is a square modulo $D/d$ and $D/d$ is a square modulo $d$. We denote by $S(D)$ the set of special divisors of $D$, and by $T(D)$ the set of all divisors of $D$. 
The set $T(D)$ has naturally the structure of a vector space over 
$\F_2$
under the operation 
\[d_1 
\odot d_2:=\frac{d_1d_2}{\gcd(d_1,d_2)^2}.\]
\begin{lemma}\label{positive divisors are F_2 vect.space}
The set $S(D)$ is a subspace of $T(D)$ over $\F_2$. 
\begin{proof}
We need to show that if $d_1,d_2$ are special then $d_1 \odot  d_2$ is special as well. 
This amount to showing firstly 
that 
if a prime $q$ divides $D$ but $q\nmid d_1 \odot d_2$ then $d_1 \odot  d_2$ is a square  $\md{q}$
and secondly that if a prime $q$ divides $d_1 \odot d_2$ then 
$D/d_1 \odot d_2$ is a square $\md{q}$.

For the proof of the first claim,
suppose that $q|D$ but $q\nmid d_1 \odot d_2$. Then either $\gcd(d_1d_2,q)=1$ or $q|\gcd(d_1,d_2)$. 
In the first case we know that, since both $d_1$ and $d_2$ are special, $d_1$ and $d_2$ are both squares $\md{q}$,
thus showing that  
$d_1 \odot d_2$
is a square $\md{q}$. 
In the second case we know that, since both $d_1$ and $d_2$ are special, $D/d_1$
and $D/d_2$ are both squares $\md{q}$. This shows that
$$\frac{D}{d_1} \frac{D}{d_2}= (d_1 \odot d_2) 
\Bigg(\frac{D}{\frac{d_1d_2}{\gcd(d_1,d_2)}}\Bigg)^2 $$ 
is 
a square $\md{q}$, hence $d_1 \odot d_2$ is a square $\md{q}$. 

Next, suppose that $q \mid d_1 \odot d_2$. 
Then, either $q \mid d_1$ and $q \nmid d_2$, or $q \mid d_2$ and $q \nmid d_1$: by symmetry 
we are allowed to focus on the former case.
Then, since both $d_1$ and $d_2$ are special, we have that both 
$D/d_1$ and $d_2$ are
squares $\md{q}$. 
Therefore $$\frac{D}{d_1}   \frac{1}{d_2}  \gcd(d_1,d_2)^2=\frac{D}{(d_1 \odot d_2)}$$ is a square $\md{q}$, 
thus concluding our proof. 
\end{proof}
\end{lemma}
Let $n$ be another square-free odd positive integer with $\gcd(n,D)=1$
and
consider the group $G_n:=(\mathbb{Z}/n\mathbb{Z})^*/(\mathbb{Z}/n\mathbb{Z})^{*2}$. One has a natural map 
$\phi_{n,D}:S(D)\to G_n$ by reducing  $\md{n}$ 
and then modulo squares. 
\begin{lemma}\label{phi is a hom}
The map $\phi_{n,D}$ is a homomorphism of $\mathbb{F}_2$-vector spaces. 
\begin{proof}
By definition we have $d_1 \odot d_2=\frac{d_1d_2}{\gcd(d_1,d_2)^2}$ and reducing this equality
$\md{n}$ and then modulo squares, the right side yields $d_1 d_2$. 
Thus $\phi_{n,D}(d_1 \odot d_2)=\phi_{n,D}(d_1)  \phi_{n,D}(d_2)$.
\end{proof}
\end{lemma}
 Observe that $S(D)$ always contains the subgroup $\{1,D\}$. It is then a consequence of the work of Fouvry and Kl\"{u}ners~\cite{MR2276261} that $S(D)/\{1,D\}$ behaves like the $2$-torsion of a random abelian $2$-group, in the sense of Cohen and Lenstra. 
In other words,
for every positive integer $j$ we have 
\[
\lim_{X \to \infty} \frac{\#\{1\leq D\leq X, D \text{ square-free}:
S(D)/\{1,D\} \cong \mathbb{F}_2^{j}\}}{\#\{1\leq D\leq X, D \text{ square-free}\}}
=
\mcl(A \in \mathcal{G}_2: A[2] \cong \mathbb{F}_2^{j})
,\]
where
$\mathcal{G}_2$ is 
a set of representatives of isomorphism classes of finite abelian $2$-groups.
The present section in addition to Theorems~\ref{m.t. on multi-moments}-\ref{m.t. on distribution}, 
\S \ref{section: special divisors thm} and \S\ref{from m.m. to distr.} are devoted to the determination of the distribution of the pair 
\[(\#S(D), \Im(\phi_{n,D})).\]

The general heuristic constructed in \S\ref{section:heuristic at 2} specializes 
to a heuristic model for this
pair, thanks to the commutative diagram
after Lemma~\ref{recognizing img-r}. However, we choose to
give here a direct presentation of this heuristic
avoiding ray
class groups.
Therefore the present section, 
Theorems~\ref{m.t. on multi-moments}-\ref{m.t. on distribution},
\S \ref{section: special divisors thm} and \S\ref{from m.m. to distr.}
are completely self-contained.

Before proceeding
we introduce a modification of $\phi_{n,D}$
which will be required in the ray class group applications in \S\ref{section:theorems on 2-parts of ray sequences}. 
Denote by $L_n$ the subgroup of $G_n$ generated by an integer which is a quadratic
non-residue modulo every prime dividing $n$
and write $\widetilde{G}_n:=G_n/L_n$. Now let $n_1,n_2$ be two integers
such that $2Dn_1n_2$ is square-free
and assume that $D$ is a square modulo $n_1$ and generates $L_{n_2} \md{n_2}$.  
Denote by $\phi_{n_1,n_2,D}$ the natural map 
$$\phi_{n_1,n_2,D}:S(D)/\{1,D\} \to G_{n_1} \times \widetilde{G}_{n_2}.
$$

Our goal is to understand the statistical
behavior of the pair 
$$(\#S(D), \Im(\phi_{n_1,n_2,D})),$$
as $D$ varies through positive square-free
integers coprime to $n_1n_2$, 
which are squares $\md{n_1}$ and non-squares modulo every prime dividing $n_2$.
There is an obvious guess: namely that, once  $\dim_{\mathbb{F}_2}(S(D)/\{1,D\})=j$ is fixed, 
then $\Im(\phi_{n_1,n_2,D})$ should distribute as the image of a random map $\phi:\mathbb{F}_2^j \to G_{n_1} \times \widetilde{G}_{n_2}$.
We formalize this guess in a more general heuristic principle. 
\begin{definition}
\label{def:jellyfish}
Consider the set $\mathcal{M}_{n_1,n_2}$ consisting of equivalence classes of pairs $(A,V)$, where $A$ is a vector space over $\F_2$  and $V$ is a vector subspace of $G_{n_1} \times \widetilde{G}_{n_2}$: declare $(A_1,V_1), (A_2,V_2)$ identified, if $A_1$ and $A_2$ have the same $\mathbb{F}_2$-dimension and $V_1=V_2$. Denote this equivalence relation by $\sim$.
Each representative pair $(\mathbb{F}_2^j,V)$
is equipped with the following 
mass,
\[ \mu((\mathbb{F}_2^j,V)):=
\mcl
(A \in \mathcal{G}_2:A[2] \cong \mathbb{F}_2^{j}) 
\frac{\#\epi_{\mathbb{F}_2}(\mathbb{F}_2^j,V)}
{\#\Hom_{\mathbb{F}_2}(\mathbb{F}_2^j,G_{n_1} \times \widetilde{G}_{n_2})}
.\]
By construction, this is a probability measure on $\mathcal{M}_{n_1,n_2}$. 
\end{definition}
Now we formulate the following.
\begin{heuristic}\label{heur for special divisor}
For any `reasonable' function $f: \mathcal{M}_{n_1,n_2} \to \mathbb{R}$ one has
$$\lim_{X \to \infty} 
\frac{
\sum_{D\leq X}
f((S(D)/\{1,D\},\Im(\phi_{n_1,n_2,D})))}{\sum_{D\leq X}1}
=\sum_{T \in \mathcal{M}_{n_1,n_2}}f(T)\mu(T)
,$$
where in both sums $D$ varies among square-free positive integers which are squares 
$\md{n_1}$ and non-squares modulo any prime divisor of $n_2$.
Furthermore, for any positive integers  $a,r$ with  $\gcd(r,an_1n_2)=1$
the same holds if we have the additional restriction $D\equiv a \md{r}$. 
\end{heuristic}
The simple case where 
$f$ is
the indicator function of an element $(\mathbb{F}_2^j,V) \in \mathcal{M}_{n_1,n_2}$
yields
the following prediction.
\begin{prediction} \label{prediction:distr. sp.div.}
We have 
$$ \lim_{X \to \infty} \frac{\#\{D\leq X, (S(D) /\{1,D\},\phi_{n_1,n_2,D}) \sim T\}}{\#\{D\leq X\}}=\mu(T),$$
where
$D$ varies among square-free positive integers which are squares $\md{n_1}$ and non-squares modulo every prime divisor of $n_2$.
\end{prediction}
This prediction will be confirmed in Theorem \ref{m.t. on distribution}.

Despite the fact that the `random variable'
$(S(D),\Im(\phi_{n_1,n_2,D}))$
does not consist of two \emph{numbers}, we achieve its distribution by means of the moment-method. 
For this we shall 
replace the pair 
$(S(D),\Im(\phi_{n_1,n_2,D}))$
by a
\emph{higher-dimensional numerical} `random variable',
which we proceed to define.
For each character 
$\chi$ in the dual of $G_{n_1} \times \widetilde{G}_{n_2}$
define 
\beq{def:twist}{ m_{\chi}(D):=\#\{d \in S(D): \chi(\phi_{n_1,n_2,D}(d))=1\}}
and recall that 
$\Im(\phi_{n_1,n_2,D})^{\perp}$ is the set of all character $\chi$ with $\chi \circ \phi_{n_1,n_2,D}$ being  trivial.
Clearly for each  
$\chi \in \Im(\phi_{n_1,n_2,D})^{\perp}$
we have
$m_{\chi}(D)=m_1(D)=\#S(D)$,
while for the remaining characters 
we have
$m_{\chi}(D)=\#S(D)/2$.
Therefore the knowledge of the pair 
$$(\#S(D),\Im(\phi_{n_1,n_2,D}))
$$
is equivalent to the knowledge of 
$$(m_{\chi}(D))_{\chi \in \widehat{G}_{n_1} \times \widehat{\widetilde{G}}_{n_2}}.$$
It will transpire  
that this shift in focus will be advantageous 
since it will allow us to
study the asymptotic behaviour of the latter vector
by the method of moments.

We conclude this section by providing 
a prediction regarding 
the
mixed
moments of $(m_\chi(D))$.
This will be later used in the proof of
Theorem \ref{m.t. on multi-moments}.

\begin{nnotation}
\label{def:sumk}
For any  
function 
$\widehat{G}_{n_1} \times \widehat{\widetilde{G}}_{n_2} \to \mathbb{Z}_{\geq 0}$,
$\chi \mapsto k_{\chi}$, 
we will use the notation 
$$ 
\b{k}:=(k_\chi)_{\chi \in \widehat{G}_{n_1} \times \widehat{\widetilde{G}}_{n_2}}
\
\
\text{ and   }
\
\
|\b{k}|_1
:=\sum_{\chi \in \widehat{G}_{n_1} \times \widehat{\widetilde{G}}_{n_2}}k_{\chi}.$$ 
\end{nnotation}
\begin{definition}
\label{def:quiteimportant} 
For any subspace $Y \subseteq \widehat{G}_{n_1} \times \widehat{\widetilde{G}}_{n_2}$, denote by $\mathbb{P}_{(k_{\chi})}(Y)$ the probability that a random subset of $\widehat{G}_{n_1} \times \widehat{\widetilde{G}}_{n_2}$ generates $Y$, where the
characters $\chi$ are chosen independently 
and
with probability $1-2^{-k_\chi}$.
\end{definition}
For any pair $(\mathbb{F}_2^j,Y)$ in $\mathcal{M}_{n_1,n_2}$, define $m_{\chi}((\mathbb{F}_2^j,Y))$ to be $2^j$ if $\chi(Y)=1$, and $2^{j-1}$ otherwise. 
Observe that if $\dim(Y)> 
\bk
$ then  $\mathbb{P}_{(k_{\chi})}(Y)=0$. Denote by $\mathcal{N}_2(j)$ the number of vector subspaces of $\mathbb{F}_2^j$. 
If $j<0$ we define $\mathcal{N}_2(j):=1$.
It is important to note that every time $\mathcal{N}_2(j)$ appears for some negative $j$
then it will always appear multiplied by zero. 
\begin{proposition}\label{heuristic comp. of multimoments}
One has that 
$$\sum_{T \in \mathcal{M}_{n_1,n_2}}
\hspace{-0,2cm}
\Big(\prod_{\chi \in \widehat{G}_{n_1} \times \widehat{\widetilde{G}}_{n_2}}m_{\chi}(T)^{k_{\chi}}\Big)
\mu(T)
=
\hspace{-0,2cm}
\sum_{W \subseteq 
\widehat{G}_{n_1} \times \widehat{\widetilde{G}}_{n_2}
}\mathbb{P}_{(k_{\chi})}(W)\mathcal{N}_2(
\bk
-\dim(W))
.$$\end{proposition}
\begin{proof}
We want to compute  
$$\sum_{(\mathbb{F}_2^j,\delta)}
\hspace{-0,2cm}
\Big(
\prod_{\chi \in 
\widehat{G}_{n_1} \times \widehat{\widetilde{G}}_{n_2} }m_{\chi}((\mathbb{F}_2^j,\delta))^{k_{\chi}})
\Big)
\mu((\mathbb{F}_2^j,\delta))
,$$
where $j$ ranges
over non-negative integers,
$\delta$ ranges over 
$\Hom(\mathbb{F}_2^j,G_{n_1} \times \widetilde{G}_{n_2})$
and 
$$\mu((\mathbb{F}_2^j,\delta))=
\frac{\mcl(A \in \mathcal{G}_2:\#A[2]=2^j)}
{\#\Hom(\mathbb{F}_2^j,G_{n_1} \times \widetilde{G}_{n_2})}.$$ 
Therefore the sum becomes
$$\sum_{V \subseteq \widehat{G}_{n_1} \times \widehat{\widetilde{G}}_{n_2} }
\sum_{j \geq 0}\frac{2^{j \bk}}{2^{\sum_{\chi \notin V}k_{\chi}}} 
\frac{
\#
\epi(\mathbb{F}_2^j,V^{\perp})}{
\#
\Hom(\mathbb{F}_2^j,G_{n_1} \times \widetilde{G}_{n_2})} 
\mcl(A \in \mathcal{G}_2:
\#A[2]=2^j)
.$$
We assume familiarity of the reader with M\"{o}bius inversion in posets,
see~\cite[Chapter 3]{MR2868112}, for example.  
Writing $\epi(\mathbb{F}_2^j,V^{\perp})$ via inclusion-exclusion on the poset of vector subspaces of $G_{n_1} \times \widetilde{G}_{n_2}$
and exchanging the
order of summation we obtain
$$\sum_{W \subset \widehat{G}_{n_1} \times \widehat{\widetilde{G}}_{n_2}} 
\Big(\sum_{V \subset W} \frac{\mu(V,W)}{2^{\sum_{\chi \not \in V}k_{\chi}}}\Big) 
\Big(\sum_{G \in \mathcal{G}_2}\#G[2]^{\bk-\dim(W)}\mcl(G)\Big)
.$$
By applying M\"{o}bius inversion with respect to the poset of vector subspaces, to the obvious relation
$$2^{-\sum_{\chi \not \in W}k_{\chi}}=\mathbb{P}_{(k_{\chi})}(V \subseteq W)=\sum_{V\subset W}\mathbb{P}_{(k_{\chi})}(V)
$$
we obtain
$$ \mathbb{P}_{(k_{\chi})}(W)=\sum_{V \subset W} \frac{\mu(V,W)}{2^{\sum_{\chi \not \in V}k_{\chi}}}.$$
On the other hand, one has that whenever $\bk-\dim(W) \geq 0$, then 
$$\sum_{G \in \mathcal{G}_2}\#G[2]^{\bk-\dim(W)}\mcl(G)=\mathcal{N}_2(\bk-\dim(W)).$$
Instead, when 
$\bk-\dim(W)<0$, we have that 
$\mathbb{P}_{(k_{\chi})}(W)=0$. In conclusion we get that the total sum equals
\begin{equation*}
\sum_{W \subseteq \widehat{G}_{n_1} \times \widehat{\widetilde{G}}_{n_2}}\mathbb{P}_{(k_{\chi})}(W)\mathcal{N}_2(\bk-\dim(W)).
\qedhere
\end{equation*}
\end{proof} 
Choosing 
$f(T)=\prod_{\chi \in \widehat{G}_{n_1} \times \widehat{\widetilde{G}}_{n_2}}m_{\chi}(T)^{k_{\chi}}$
in Heuristic assumption~\ref{heur for special divisor}
suggests the following prediction by means of
Proposition~\ref{heuristic comp. of multimoments}.
\begin{prediction}\label{prediction:m.m. special divisors}
We have 
\[
\lim_{X \to \infty} 
\frac{\sum_{D\leq X}\prod_{\chi \in \widehat{G}_{n_1} \times \widehat{\widetilde{G}}_{n_2} }m_{\chi}(D)^{k_{\chi}}}
{
\sum_{D\leq X} 1
}
=
2^{\bk}
\hspace{-0,5cm}
\sum_{W \subseteq \widehat{G}_{n_1} \times \widehat{\widetilde{G}}_{n_2}}
\hspace{-0,2cm}
\mathbb{P}_{(k_{\chi})}(W)\mathcal{N}_2(\bk-\dim(W))
,\]
where in both sums
$D$ varies among square-free positive integers which are squares   $\md{n_1}$ and non-squares modulo every prime divisors of $n_2$.
\end{prediction}
A version of Prediction~\ref{prediction:m.m. special divisors}
with an explicit error term 
is
proved in Theorem~\ref{m.t. on multi-moments}.
This prediction has a noteworthy
feature: it realizes the $(k_{\chi})$-mixed moments of $(m_{\chi}(D))$ as an average over all subspaces of $\widehat{G}_{n_1} \times \widehat{\widetilde{G}}_{n_2}$ of ordinary moments of $\#S(D)$
and in doing so,
it suggests the first step of the proof of Theorem~\ref{m.t. on multi-moments}, see~\eqref{eq:cough1234}.
\section{Main theorems on the $2$-part of ray class sequences} 
\label{section:theorems on 2-parts of ray sequences}
Throughout the section we keep the notation used in \S\ref{section:heuristic at 2}. 
We begin by stating 
Theorems~\ref{theorem:multi-moments of ray class seq},\ref{theorem:distribution of delta maps} and \ref{theorem: joint distribution of 4-ranks}
that corroborate
Predictions~\ref{pred:croisant},\ref{pred:croisant2} and~\ref{pred:croisant3}
when 
$\disc(K) \equiv  1  \md{4}$.
We restrict our attention to the cases with 
$\disc(K) \equiv  1  \md{4}$
only for the sake of brevity,
the remaining case being amenable to a 
similar analysis.
Our main task in this section will then be to 
reduce 
Theorems~\ref{theorem:multi-moments of ray class seq}, \ref{theorem:distribution of delta maps}, \ref{theorem: joint distribution of 4-ranks}
and~\ref{theorem: joint distribution of 4-ranks, explicit version}
that are
about ray class groups
to Theorems~\ref{m.t. on multi-moments}
and~\ref{m.t. on distribution}
which 
regard only
special divisors.
\begin{theorem} 
\label{theorem:multi-moments of ray class seq}
For any 
$\beta\in \R$ satisfying 
$0<\beta<\min\{2^{-\bk},\phi(c)^{-1}\}$
we have 
$$ 
\frac{
\sum_{-\disc(K)\leq X}\prod_{\chi \in \widehat{G}_{n_1} \times \widehat{\widetilde{G}}_{n_2} }m_{\chi}(\delta_2(K))^{k_{\chi}}}
{\sum_{-\disc(K)\leq X}1}
=
\hspace{-0,3cm}
\sum_{V \subseteq \widehat{Im(g_R)}}
\hspace{-0,3cm}
\mathbb{P}_{(k_{\chi})}(V)\mathcal{N}_2(\bk-\dim(V))
+O((\log X)^{-\beta})
,$$ 
where in both sums $K$ varies among imaginary quadratic number fields of type $R$, having $\disc(K) \equiv  1  \md{4}$
and the implied constant
depends at most on $c$ 
and $(k_\chi)_\chi$.
\end{theorem}
\begin{theorem} \label{theorem:distribution of delta maps}
We have 
\begin{align*}
&\lim_{X \to \infty}\frac{\#\{K: -\disc(K)\leq X,\#(2\cl(K))[2]=2^j \ \text{and} \ \Im(\delta_2(K))=Y\}}{\#\{K: -\disc(K)\leq X \}}
\\=&
\mu_{\emph{CL}}(G \in \mathcal{G}_2:\#G[2]=2^j) 
\frac{\#\epi_{\mathbb{F}_2}(\mathbb{F}_2^j,Y)}{\#\Hom_{\mathbb{F}_2}(\mathbb{F}_2^j,\Im(g_R))},
\end{align*}
where 
$K$ varies among imaginary quadratic number fields 
with $\disc(K) \equiv 1  \md{4}$
and 
of type $R$.
\end{theorem}
Recall
the definition of $W_R$ in~\eqref{def:wr}
and the definition of the map $g_R$
before the statement of Proposition~\ref{delta maps for ext tilde}.
\begin{theorem} \label{theorem: joint distribution of 4-ranks}
We have 
\begin{align*}
&\lim_{X \to \infty} \frac{\#\{K: -\disc(K)\leq X, \rk_4(\cl(K))=j_1, \rk_4(\cl(K,c))=j_2\}}{\#\{K:-\disc(K)\leq X \}}
\\=&
\mu_{\emph{CL}}(G \in \mathcal{G}_2:\dim_{\mathbb{F}_2}(G[2])=j_1)
\frac{\#\{\phi \in \Hom_{\mathbb{F}_2}(\mathbb{F}_2^{j_1},\Im(g_R)): \rk(\phi)=\rk_{2}(W_R)-(j_2-j_1)\}}{\#\Hom_{\mathbb{F}_2}(\mathbb{F}_2^{j_1},\Im(g_R))}
,\end{align*}
where 
$K$ varies among imaginary quadratic number fields 
with $\disc(K) \equiv 1  \md{4}$
and 
of type $R$.
\end{theorem}
We will prove a stronger version of
Theorems \ref{theorem:multi-moments of ray class seq},~\ref{theorem:distribution of delta maps}
and~\ref{theorem: joint distribution of 4-ranks}.
Namely,
the fact that we deal with progressions $a\md{q}$ in Theorems~\ref{m.t. on multi-moments} and~\ref{m.t. on distribution}
yields results 
analogous
to the ones in Theorems~\ref{theorem:multi-moments of ray class seq}, \ref{theorem:distribution of delta maps}
and~\ref{theorem: joint distribution of 4-ranks}
when one imposes 
finitely many unramified local conditions
at primes independent of $c$
on the discriminants $D(K)$.
This supports the point of view in Wood's recent
work~\cite{mwmw}
that local conditions on the quadratic field do not affect the distribution of
class groups,
with the obvious modification that for ray class groups 
such conditions must be taken independently of the primes dividing $c$.

We proceed to 
restate Theorem~\ref{theorem: joint distribution of 4-ranks}
in a more explicit way.
Recalling that $c$ is square-free
we let $n_1(R)$ be the 
product of the prime divisors of
$c$ which are either $3\md{4}$ and inert in $R$, 
or $1\md{4}$ and split.  
Furthermore, let $n_2(R):=c/n_1(R),$ 
this is the product of the prime divisors of $c$
that are $3\md{4}$ and split in $R$.
Recall that 
\[
\frac{\eta_{\infty}(2)}{\eta_{j_1}(2)^2 2^{j_1^2}}
=
\mu_{\emph{CL}}(G \in \mathcal{G}_2:\dim_{\mathbb{F}_2}(G[2])=j_1)
.\]
\begin{theorem} 
\label{theorem: joint distribution of 4-ranks, explicit version}
We have 
\begin{align*}
&\lim_{X \to \infty} \frac{\#\{K: -\disc(K)\leq X, \rk_4(\cl(K))=j_1, \rk_4(\cl(K,c))=j_2\}}{\#\{K:-\disc(K)\leq X \}}
\\=&
\frac{\eta_{\infty}(2)}{\eta_{j_1}(2)^2   2^{j_1^2}}
\frac{\#\{\phi \in \Hom_{\mathbb{F}_2}(\mathbb{F}_2^{j_1},G_{n_1(R)} \times \widetilde{G}_{n_2(R)}): \rk(\phi)=\rk_{2}(W_R)-(j_2-j_1)\}}{\#\Hom_{\mathbb{F}_2}(\mathbb{F}_2^{j_1},G_{n_1(R)} \times G_{n_2(R)})}
,\end{align*}
where 
$K$ varies among imaginary quadratic number fields 
with $\disc(K) \equiv 1  \md{4}$
and 
of type $R$.
\end{theorem}
The congruence conditions  $\md{4}$ related to the definition of $n_1(R)$ and $n_2(R)$ in
Theorem~\ref{theorem: joint distribution of 4-ranks, explicit version}
are analogous to the
congruences $\md{3}$ for the primes $l$ appearing in the first part of Varma's Theorem~\ref{thm:varvarvar}.

Our next goal is to realise the 
$\delta_2$-map 
$$\delta_2(\mathbb{Q}(\sqrt{-D}))
: (2\cl(\mathbb{Q}(\sqrt{-D})))[2] \to \Im(g_R)
$$
with the map on special divisors introduced in \S\ref{section:special divisors},
$$\phi_{n_1(R),n_2(R),D}:\frac{S(D)}{\{1,D\}} \to G_{n_1(R)} \times \widetilde{G}_{n_2(R)}.
$$  
\subsection{Realizing $\delta_2(\mathbb{Q}(\sqrt{-D}))$ as $\phi_{n_1(R),n_2(R),D}$} 
\label{realizing delta maps}
Let $D$ be a square-free positive integer with $D\equiv 3\md{4}$.
and 
denote its 
its prime factorization
by
$D=p_1\cdots p_j$. 
Let $\mathfrak{p}_1,\ldots,\mathfrak{p}_j$ be
the corresponding prime ideals in $\mathbb{Q}(\sqrt{-D})$, i.e. $\mathfrak{p}_i^2=(p_i)$). 
Recall that $\cl(\mathbb{Q}(\sqrt{-D}))[2]$ is generated by $\mathfrak{p}_1,\ldots,\mathfrak{p}_j$ subject only
to the
relation 
$\mathfrak{p}_1 \cdots \mathfrak{p}_j=(\sqrt{-D})$. 
For any $b$ positive divisor of $D$, denote by $\mathfrak{b}$ the ideal of $\mathbb{Q}(\sqrt{-D})$ with $\mathfrak{b}^2=(b)$. 
Let us now
recall from~\cite[Lem.16]{MR2276261} that given
a positive divisor $b$ of
$D$, 
we have 
$\mathfrak{b} \in 2\cl(\mathbb{Q}(\sqrt{-D})$ if and only if $b \in S(D)$. 
The assignment $\mathfrak{b} \mapsto b$ 
gives an
isomorphism
$$(2\cl(\mathbb{Q}(\sqrt{-D})))[2] \cong 
S(D)/\{1,D\}
.$$
Indeed, from the proof of~\cite[Lem.16]{MR2276261}, 
we know that 
$b\in S(D)$
if and only if there exists
a primitive element (i.e. not divisible by any $m \in \mathbb{Z}_{\geq 2}$)
$\alpha \in \mathcal{O}_{\mathbb{Q}(\sqrt{-D})}$  
and $w \in \mathbb{Z}_{\neq 0}$ such that 
\beq{eq:conicst}
{
bw^2=\n_{\mathbb{Q}(\sqrt{-D})/\mathbb{Q}}(\alpha)
.}
In that case the factorization of
$(\alpha)$ 
gives an integral ideal $h(\mathfrak{b})$ such that 
$(\alpha)=h(\mathfrak{b})^2\mathfrak{b}$.
We rewrite this as 
$\mathfrak{b}  (\alpha/b)=h(\mathfrak{b})^2$
and
observe that this shows in particular
that 
$\mathfrak{b}\in 2\cl(\mathbb{Q}(\sqrt{-D}))$.

By weak approximation for conics, one has that such an $\alpha$ can be found with $(\alpha,c)=1$,
i.e. a primitive point on~\eqref{eq:conicst} such that $\gcd(w,c)=1$.
Therefore both $(\alpha),h(\mathfrak{b})$ are coprime to $(c)$. 
Therefore the fractional ideal $\mathfrak{b} (\frac{\alpha}{b})$ can be employed as a lifting of $\mathfrak{b}$ to $2\cl(\mathbb{Q}(\sqrt{-D}),c)$. Therefore the definition of the 
$\delta_2$-map gives us that 
$$ \delta_2(\mathbb{Q}(\sqrt{-D}))(\mathfrak{b})=b\frac{\alpha^2}{b^2}.$$
However squares of integers
in $W_R/2W_R$ give rise to the trivial element, 
therefore by~\eqref{eq:conicst}
we obtain that 
$\delta(\mathfrak{b})=g_R(\alpha)$
Recalling that 
$N(\cdot)$ is the norm-function with respect to the $C_2$-action prescribed to $R^{*}/\langle -1 \rangle$
we see that $g_R(\alpha)=\alpha^2  N(\alpha)$.
Next, we provide a more concrete description of $\Im(g_R)$.  
The proof of the following result is straightforward
and therefore omitted.
\begin{lemma} \label{recognizing img-r}
There is an isomorphism 
$\phi_R: \Im(g_R) \to G_{n_1(R)} \times G_{n_2(R)}$
such that $$\phi_R(g_R(x))=N(x)$$ for every $x \in \frac{R^{*}}{\langle -1 \rangle}[2^{\infty}]$.
\end{lemma} 
Since $N(\alpha)=bw^2$ and $w^2$ is trivial in $W_R/2W_R$, we get a commutative diagram
$$
\begin{array}{ccc}  (2\cl(\mathbb{Q}(\sqrt{-D}))[2] &\overset{\delta}\to& \Im(g_R)  
\\ \downarrow && \downarrow \phi_R\\ \frac{S(D)}{\{1,D\}} &\underset{\phi_{n_1,n_2,D}}\to& G_{n_1(R)} 
\times \widetilde{G}_{n_2(R)}  
\end{array}
$$
where the vertical rows are isomorphisms. This gives us
precisely the realization of the $\delta_2$-map in terms of special divisors that we were looking for.
\subsection{Reduction to special divisors}
Our next 
result holds for integers 
$a,q,n_1,n_2$
satisfying
\beq{eq:donaldtrump}
{4n_1n_2\text{ divides } q,
a\equiv 3 \md{4},
\gcd(a,q)=1
,}
\beq{eq:donaldtrump2}
{ 
a \text{ is a square} \md{n_1}
}
and
\beq{eq:donaldtrump3}
{p \text{ prime}, p\mid n_2
\Rightarrow
a \text{ is a non-square} \md{p}.}
\begin{theorem}\label{m.t. on multi-moments}
Let $a,q,n_1,n_2$ 
be
positive integers
satisfying~\eqref{eq:donaldtrump},
~\eqref{eq:donaldtrump2}
and
~\eqref{eq:donaldtrump3}. 
Then for every $\delta \in (0,2^{-\bk})$ we have 
$$\frac{
\sum_{D\leq X}\prod_{\chi \in \widehat{G}_{n_1} \times \widehat{\widetilde{G}}_{n_2} }m_{\chi}(D)^{k_{\chi}}
}{\sum_{D\leq X}1}
-2^{\bk 
}
\Big(
\hspace{-0,2cm}
\sum_{W \subseteq \widehat{G}_{n_1} \times \widehat{\widetilde{G}}_{n_2}}
\hspace{-0,3cm}
\mathbb{P}_{(k_{\chi})}(W)\mathcal{N}_2(\bk-\dim(W))\Big)
\ll
(\log X)^{-\delta}
,$$
where in both sums
$D$
varies among
square-free positive integers which are congruent to $a\md{q}$
and the implied constant
depends at most 
on $a,q,n_1,n_2,\delta$  and $(k_\chi)_\chi$.
\end{theorem}
This proves Prediction \ref{prediction:m.m. special divisors}
with an explicit error term.

Recall Definition~\ref{def:jellyfish}.
We shall use Theorem~\ref{m.t. on multi-moments} in \S\ref{from m.m. to distr.} to deduce the following.
\begin{theorem}\label{m.t. on distribution}
Let $a,q,n_1,n_2$ 
be
positive integers
satisfying~\eqref{eq:donaldtrump},
~\eqref{eq:donaldtrump2}
and
~\eqref{eq:donaldtrump3}. 
Then 
$$\lim_{X \to \infty} \frac{\#\{D\leq X, (S(D) /\{1,D\},\phi_{n_1,n_2,D}) \sim T\}}{\#\{D\leq X\}}=\mu(T),$$
where 
$D$ varies among
positive
square-free
integers satisfying $D\equiv a\md{q}$. 
\end{theorem}
This confirms the Prediction~\ref{prediction:distr. sp.div.}.

We are finally in place to explain why
Theorems~\ref{m.t. on multi-moments}
and~\ref{m.t. on distribution}
imply
Theorems~\ref{theorem:multi-moments of ray class seq}, \ref{theorem:distribution of delta maps}, \ref{theorem: joint distribution of 4-ranks}
and~\ref{theorem: joint distribution of 4-ranks, explicit version}.
Owing to the final diagram of the previous subsection, we have the following implications. 
Theorems~\ref{theorem:distribution of delta maps}, \ref{theorem: joint distribution of 4-ranks}
and~\ref{theorem: joint distribution of 4-ranks, explicit version}
follow immediately from Theorem~\ref{m.t. on distribution}
because the family of fields $K$
that 
are strongly of type $R$
has zero proportion.

To
deduce
Theorem~\ref{theorem:multi-moments of ray class seq}
from
Theorem~\ref{m.t. on multi-moments}
recall the definition of $E(X)$ given prior to~\eqref{eq:acacac}
and that 
$m_\chi(\delta_2(K))$ coincides with $m_\chi(-\disc(K))$ if $\disc(K) \notin E(X)$
and that it vanishes otherwise.
We thus obtain
\beq{eq:xixi}
{\sum_{-\disc(K)\leq X}\prod_{\chi \in \widehat{G}_{n_1} \times \widehat{\widetilde{G}}_{n_2} }m_{\chi}(\delta_2(K))^{k_{\chi}}
-
\sum_{D\leq X}\prod_{\chi \in \widehat{G}_{n_1} \times \widehat{\widetilde{G}}_{n_2} }m_{\chi}(D)^{k_{\chi}}
=-
\sum_{D\in E(X)}
\prod_{\chi \in \widehat{G}_{n_1} \times \widehat{\widetilde{G}}_{n_2} }m_{\chi}(D)^{k_{\chi}}
.}
Fixing any $\gamma \in (0,1/\phi(c))$  
we can pick a positive integer $p'$ which satisfies
$\gamma \phi(c)<1-1/p'<1$
and define $q'$ via
$1/p'+1/q'=1$.
Using 
H\"{o}lder's inequality
we see that the quantity in~\eqref{eq:xixi}
has modulus 
\begin{align*}
\sum_{D\in E(X)}
\prod_{\chi \in \widehat{G}_{n_1} \times \widehat{\widetilde{G}}_{n_2} }m_{\chi}(D)^{k_{\chi}}
&=
\sum_{D\leq X}\mathbf{1}_{E(X)}(D)
\Big(\prod_{\chi \in \widehat{G}_{n_1} \times \widehat{\widetilde{G}}_{n_2} }m_{\chi}(D)^{k_{\chi}}\Big)
\\&\leq 
\Big(\sum_{D\leq X}\mathbf{1}_{E(X)}(D)^{q'}\Big)^{1/q'}
\Big(\sum_{D\leq X}\prod_{\chi \in \widehat{G}_{n_1} \times \widehat{\widetilde{G}}_{n_2}}m_{\chi}(D)^{p'k_{\chi}}\Big)^{1/p'}
\\&=
E(X)^{1/q'}
\Big(\sum_{D\leq X}\prod_{\chi \in \widehat{G}_{n_1} \times \widehat{\widetilde{G}}_{n_2}}m_{\chi}(D)^{p'k_{\chi}}\Big)^{1/p'}
.
\end{align*}
Observe that 
the obvious
bound
$m_{\chi}(D) \leq \#S(D)$ 
shows that the second sum is 
\[\leq 
\sum_{D\leq X}
\#S(D)^{p'\bk}
\]
hence by~\cite[Th.9]{MR2276261}
it is $O_{p',\b{k}}(X)$.
Using~\eqref{eq:acacac}
we conclude that 
the quantity in~\eqref{eq:xixi}
is
\[\ll
\Big(\frac{X}{(\log X)^{1/\phi(c)}}\Big)^{1/q'} X^{1/p'} 
=\frac{X}{(\log X)^{1/(q'\phi(c))}}
\ll
\frac{X}{(\log X)^{\gamma}}
,\]
This concludes our argument that shows that  
Theorem~\ref{m.t. on multi-moments} 
implies
Theorem~\ref{theorem:multi-moments of ray class seq}.

\section{Main theorems on special divisors} \label{section: special divisors thm}
This section is devoted to the proof of Theorem~\ref{m.t. on multi-moments}. 
\subsection{Pre-indexing trick} \label{pre-indexing}
In the present subsection we reduce Theorem~\ref{m.t. on multi-moments}
into a statement that can be proved with the method of Fouvry and Kl\"{u}ners.
Recall the definition of the set of special divisors $S(D)$
given in the beginning of \S\ref{section:special divisors}.
For a  
character $\chi \in \widehat{G}_{n_1} \times \widehat{\widetilde{G}}_{n_2}$ we bring into play the sum
\begin{equation}
\label{def:art1080}
A_{\chi}(D):=\sum_{a'b'=D}
\chi(a')
\Big(
\sum_{c' \mid b'}
\Big(\frac{a'}{c'}\Big)
\Big)
\Big(
\sum_{d' \mid a'}
\Big(\frac{b'}{d'}\Big)
\Big)
\end{equation}
and let $A(D):=A_1(D)$. 
By definition~\eqref{def:twist}
we see that 
$m_{\chi}(D)$
is the cardinality of elements 
$a' \in S(D)$
such that $\chi(a')=1$.
Detecting the latter condition via  
$(1+\chi(a'))/2$
we obtain
$$m_{\chi}(D)=
2^{-\omega(D)}
\frac{(A(D)+A_\chi(D))}
{2}
.$$
Recalling Notation~\ref{def:sumk}  
we obtain 
\beq{eq:cough1234}{
\prod_{\chi \in \widehat{G}_{n_1} \times \widehat{\widetilde{G}}_{n_2} }
\hspace{-0,3cm}
m_{\chi}(D)^{k_{\chi}}
=
2^{-\bk  \omega(D)} 
\frac
{\prod_{\chi \in \widehat{G}_{n_1} \times \widehat{\widetilde{G}}_{n_2} }(A(D)+A_{\chi}(D))^{k_{\chi}}}
{2^{\bk}}
.}
Letting 
$|(i_\chi)|_1$ be the 
$\ell^1$-norm of the vector $(i_\chi)_\chi$
 we see that the 
right side equals
\[2^{-\bk\omega(D)}
\sum_{\substack{(i_\chi)_{\chi}\\0\leq i_\chi \leq k_\chi}}
\frac{\lambda_{(i_\chi)}}{2^{\bk}}
A(D)^{\bk-|(i_\chi)|_1}   \prod_{\chi \in \widehat{G}_{n_1} \times \widehat{\widetilde{G}}_{n_2}}
A_{\chi}(D)^{i_{\chi}}
\]
for some integers
$\lambda_{(i_\chi)}$.
To each vector $(i_\chi)$ we attach the space $$
Y_{(i_\chi)}
:=\langle \{\chi:i_\chi\neq 0\}\rangle\subseteq
\widehat{G}_{n_1} \times \widehat{\widetilde{G}}_{n_2}
$$
and 
recalling 
Definition~\ref{def:quiteimportant}
we see that for a fixed subspace $Y\subseteq \widehat{G}_{n_1} \times \widehat{\widetilde{G}}_{n_2}$
we have
\[
\sum_{\substack{(i_\chi):Y_{(i_\chi)}=Y \\ 0\leq i_\chi \leq k_\chi}}
\frac{\lambda_{(i_\chi)}}{2^{\bk}}
=\mathbb{P}_{(k_{\chi})}(Y)
.\]
Hence
Theorem~\ref{m.t. on multi-moments}
would follow from proving
that for any $\epsilon>0$,
any integers
$a,q,n_1,n_2$ 
satisfying~\eqref{eq:donaldtrump},
~\eqref{eq:donaldtrump2}
and
~\eqref{eq:donaldtrump3},
any $B \subseteq \widehat{G}_{n_1} \times \widehat{\widetilde{G}}_{n_2}-\{1\}$
and any choice of a function $i:B \to \mathbb{Z}_{>0}$ with
$i_{\chi} \leq k_{\chi}$, one has that
\begin{equation}
\begin{aligned}
\label{eq:propo7}
&\sum_{D\leq X} 2^{-\bk\omega(D)}A(D)^{\bk-\sum_{\chi \in B}i_{\chi}}  \prod_{\chi \in B}A_{\chi}(D)^{i_{\chi}} \\
=&2^{\bk}
\mathcal{N}_2(\bk-\dim(Y_{(i_\chi)}))
\Big(
\sum_{D\leq X}1
\Big)
+O(X(\log X)^{\epsilon-2^{-\bk}})
,\end{aligned}
\end{equation} 
where in both sums
$D$
varies among
positive
square-free  integers which are congruent to $a\md{q}$.
Here
$\mathcal{N}_2(h)$ denotes as usual the number of vector subspaces of $\mathbb{F}_2^{h}$.
To prove~\eqref{eq:propo7} 
we will use 
the approach in the proof of~\cite[Th.6]{MR2276261}.
In the present notation their result corresponds to the case 
$B=\emptyset$
in~\eqref{eq:propo7}.

\subsection{Indexing trick}\label{indexing}
We begin by performing the following change of variables in~\eqref{def:art1080}, 
\[
a'=D_{10}D_{11}, b'=D_{00}D_{01},c'=D_{00}, d'=D_{11}.
\]
Letting
$\Phi_1(\b{u},\b{v}):=(\b{u}_1+\b{v}_1)(\b{u}_1+\b{v}_2)$
and
$\Psi(\b{u}):=\b{u}_1$
we can thus conclude that 
$$A_{\chi}(D)=
\sum_{D=D_{10}D_{11}D_{00}D_{01}}
\prod_{(\b{u},\b{v}) \in (\mathbb{F}_2^2)^2}
\Big(
\frac{D_\b{u}}{D_\b{v}}
\Big)^{\Phi_1(\b{u},\b{v})}
\prod_{\b{u} \in \mathbb{F}_2^2}
\chi(D_\b{u})^{\Psi(\b{u})}
.$$
Next, if $\langle B \rangle$ is not the zero
subspace 
we
choose 
a basis $T\subset B$
of $\langle B \rangle$.
Now suppose
we choose in each factor of
$$A(D)^
{
\bk
-\sum_{\chi \in B}i_{\chi}}  
 \prod_{\chi \in B}A_{\chi}(D)^{i_{\chi}}
$$ 
a decomposition of $D$ as follows,
$$D=\prod_{\b{u}^{(1)} \in \mathbb{F}_2^2}D_{\b{u}^{(1)}}^{(1)}=\ldots=
\prod_{\b{u}^{(\bk)} \in \mathbb{F}_2^2}D_{\b{u}^{\bk}}^{(\bk)}
.$$
We change variables and write $D_{\b{u}^{(1)},\ldots,\b{u}^{(\bk)}}:=
\gcd(D_{\b{u}^{(1)}}^{(1)},\ldots,D_{\b{u}^{(\bk)}}^{(\bk)})$, where 
one can reconstruct the old variables with
the help of  
$$D_{\b{u}^{(\ell)}}^{(\ell)}
=
\prod_{\substack{1\leq n \leq \bk\\n\neq \ell}}
\prod_{\bu^{(n)} \in \F_2^2}
D_{\b{u}^{(1)},\ldots,\b{u}^{(\ell)},\ldots,\b{u}^{(\bk)}} 
$$
as in 
~\cite[Eq.(23)]{MR2276261}.
Thus we can write
$$A(D)^{\bk-\sum_{\chi \in B}i_{\chi}}\prod_{\chi \in B}
A_{\chi}(D)^{i_{\chi}}
=
\sum_{\prod_{\b{u} \in \mathbb{F}_2^{2\bk}}D_\b{u}=D} 
\Big(
\prod_{\b{u},\b{v} \in \mathbb{F}_2^{2\bk}}\Big(\frac{D_\b{u}}{D_\b{v}}\Big)^{\Phi_{\bk}(\b{u},\b{v})}
\Big)
\Big(\prod_{\b{u} \in \mathbb{F}_2^{2\bk}}\prod_{\chi \in T}\chi(D_\b{u})^{\Psi_{\chi}(\b{u})}
\Big)
,$$
where $$\Phi_{\bk}(\b{u},\b{v}):=\sum_{j=1}^{\bk}\Phi_1(\b{u}^{(j)},\b{v}^{(j)})$$ and 
$\Psi_{\chi}$ are linear maps from $\mathbb{F}_2^{2\bk}$ to $\mathbb{F}_2$, which we next describe. 
Decompose
$$\mathbb{F}_2^{2\bk}=\mathbb{F}_2^{2\bk-2\sum_{\chi \in B}i_{\chi}}\times \prod_{\chi \in B} \mathbb{F}_2^{2i_{\chi}}$$
and 
we denote
a vector in this space as $\b{u}:=(\b{u}_0,(\b{u}^{(\chi)})_{\chi \in B})$, where 
$
\b{u}^{(\chi)}
:=(\b{u}^{(\chi)}_1,\ldots,\b{u}^{(\chi)}_{i_{\chi}})$
and 
for every $j$ we have 
$\b{u}^{(\chi)}_j \in \mathbb{F}_2^2$.
Next, write
$$\Psi_{\chi}^{'}(\b{u})=\sum_{j=1}^{i_\chi}\Psi(\b{u}^{(\chi)}_j)
$$
and note that 
we have 
\beq{eq:milkespresso}{
\Psi_{\chi}(\b{u})=\sum_{\chi' \in B_{\chi}}\Psi_{\chi'}'(\b{u}),
}
where $B_{\chi}$ denotes the set of characters $\chi' \in B$, such that $\chi$ is used in writing $\chi'$ in the basis $T$.
In particular, this implies that  $\chi \in B_{\chi}$. 
The construction of 
$\Psi_{\chi}$ 
depends on $T$ and $(i_\chi)$, but we suppress this dependency to simplify the notation.

Let us observe that 
there are 
$\#T=
\dim(\langle B \rangle)$
many
linear
maps
$\Psi_{\chi}$ 
and that they are 
independent.
Indeed, given $\chi \in T$, all maps
$\Psi_{\chi'}$ with $\chi' \in T-\{\chi\}$ vanish on the vectors $\b{u}$ with $\b{u}^{({\widetilde{\chi}})}=\b{0}$ 
for each $\widetilde{\chi} \neq \chi$, while $\Psi_{\chi}$ evaluated in such $\b{u}$ equals
$\Psi_{\chi}'(\b{u}^{(\chi)})$,
which does not vanish identically. 

We can 
therefore 
rewrite the first sum over $D$ in~\eqref{eq:propo7} as 
\begin{equation}
\begin{aligned}
\label{eq:propoftft}
&\sum_{D\leq X} 2^{-\bk\omega(D)}A(D)^{\bk-\sum_{\chi \in B}i_{\chi}}  \prod_{\chi \in B}A_{\chi}(D)^{i_{\chi}} \\
= 
&\sum_{(D_\b{u})}
\Big(
\prod_{\b{u} \in \mathbb{F}_2^{2\bk}}
2^{-\bk\omega(D_\b{u})}
\Big)
\Big(
\prod_{\b{u},\b{v} \in \mathbb{F}_2^{2\bk}}
\Big(\frac{D_\b{u}}{D_\b{v}}\Big)^{\Phi_{\bk}(\b{u},\b{v})}
\Big)
\Big(
\prod_{\b{u} \in \mathbb{F}_2^{2\bk}}
\prod_{\chi \in T}\chi(D_\b{u})^{\Psi_{\chi}(\b{u})}
\Big)
,\end{aligned}
\end{equation}
where the second sum is over positive integers
$D_{\b{u}}$
such that 
$\prod_{\b{u} \in \mathbb{F}_2^{2\bk}}D_\b{u}$
varies among
positive
square-free integers which are congruent to $a\md{q}$
and at most $X$.

Our goal in \S\S\ref{ss:ff}-\ref{s:simplefcn} is to prove 
an asymptotic for the sum over $D_\b{u}$ in~\eqref{eq:propoftft} 
under the assumptions on the integers 
$a,q,n_1,n_2$
in Theorem~\ref{m.t. on multi-moments}. 
For
a real number $X>1$
we bring into play the 
following subset of $\N^{4^{\bk}}$,
\beq{def:DX}{
\hspace{-0,2cm}
\c{D}(X,{\bk};q,a)
\hspace{-0,1cm}
:=
\hspace{-0,1cm}
\left\{(D_\bu)_\bu\in \N^{4^{\bk}}
\hspace{-0,1cm},
\bu
\hspace{-0,1cm}=
\hspace{-0,1cm}
(\bu^{(1)},\ldots,\bu^{({\bk})})
\in (\F_2^2)^{\bk}
\hspace{-0,1cm}
:
\begin{array}{l}
\prod_{\bu}D_\bu \ \text{is square-free}, \\
\text{bounded by } X  \ \text{and} \\
\text{congruent to} \ a\md{q}
\end{array}
\right\}.
}
We are interested in asymptotically evaluating the succeeding average,
\[
S_{\boldsymbol{\chi}}(X,\bk;q,a)
\hspace{-0,1cm}
:=
\hspace{-0,5cm}
\sum_{\substack{(D_\bu) \in \c{D}(X,\bk;q,a)
}} 
\hspace{-0,5cm}
2^{-\bk\omega(D)}
\Bigg(\prod_{\b{u},\b{v} \in (\mathbb{F}_2^2)^{\bk}}\Big(\frac{D_\bu}{D_\bv}\Big)^{\Phi_{\bk}(\bu,\bv)}\Bigg)
\hspace{-0,1cm}
\Bigg(\prod_{\bu \in (\mathbb{F}_2^2)^{\bk}}\prod_{\chi \in T}\chi(D_\bu)^{\Psi_{\chi}(\bu)}\Bigg)
\]  
and in doing so
we shall not keep track of the dependence of the implied constants
on
$T,(i_\chi),\b{k},\boldsymbol{\chi},a,q,n_1,n_2$. 
The sum $S_{\boldsymbol{\chi}}$ also depends on $(i_\chi)$ and the choice of $T$ but we suppress this in the notation. 
The function $S_{\boldsymbol{\chi}}$ should be compared with~\cite[Eq.(26)]{MR2276261};
we will verify 
in \S\ref{ss:ff}
that 
the presence of the characters ${\boldsymbol{\chi}}$ 
does not 
affect the 
analysis of 
Fouvry--Kl\"uners~\cite{MR2276261}
in the error term
and we shall 
see in
\S\S\ref{s:themte}-\ref{s:simplefcn}
how their presence influences the main term. 
\subsection{The four families of sums of Fouvry and Kl\"{u}ners}
\label{ss:ff}
We begin by restricting the summation in $S_{\boldsymbol{\chi}}(X,\bk;q,a)$ to variables having a suitably small number of prime factors as in~\cite[\!\!\S 5.3]{MR2276261}.  
Letting
$\Omega:=2^{\bk+1}{\bk}^{-1}\log \log X$
we shall study the contribution, say $\Sigma_1$, 
towards $S_{\boldsymbol{\chi}}(X,\bk;q,a)$ of elements not fulfilling
\beq{def:ome}{
\omega(D_\bu)\leq \Omega, \text{ for all } \bu \in \F_2^{2\bk}
.}
Writing $m=\prod_\bu D_\bu$ and bounding each character by $1$
provides us with
\[\Sigma_1
\ll
\sum_{m\leq X}\frac{\mu(m)^2}{\tau(m)^{\bk}}
\sum_{\substack{m_1\cdots m_{4^{\bk}}=m\\\omega(m_1)>\Omega}}1
\leq
4^{-{\bk}\Omega}
\sum_{m\leq X}
\frac{\mu(m)^2}{\tau(m)^{\bk}}
\sum_{\substack{m_1\cdots m_{4^{\bk}}=m}}
4^{{\bk}\omega(m_1)}
.\]
Invoking~\cite[Eq.(1.82)]{iwko} to bound the sum over $m$ 
makes the following estimate available,
\beq{eq:sqrf}
{\Sigma_1
\ll
X
(\log X)^{-1-2^{{\bk}+1}\log(4/\mathrm{e})-2^{{\bk}}}
.}
We continue in the footsteps laid out in~\cite[\S 5.4]{MR2276261},
where four families of elements in $\N^{4^{\bk}}$ are shown to make a negligible 
contribution
towards 
a quantity that resembles $S_{\boldsymbol{\chi}}(X,\bk;q,a)$.
Using the trivial bound 
\beq{eq:trivbnd78}
{
\Bigg|\prod_{\bu \in (\mathbb{F}_2^2)^{\bk}}\prod_{\chi \in T}\chi(D_\bu)^{\Psi_{\chi}(\bu)}\Bigg|
\leq 1
}
allows us 
to 
adopt
in a straightforward manner
the arguments leading to~\cite[Eq.(34),(39)]{MR2276261}
and we proceed to briefly explain how. 
Let 
\beq{def:delta}
{
\Delta:=1+(\log X)^{-2^{\bk}}
}
and let $A_\bu$ denote numbers of the form $\Delta^m$ where $m\in \Z_{\geq 0}$.
For $\b{A}=(A_\bu)_{\bu \in (\F_2^2)^{\bk}}$ we let 
\[
S_{\boldsymbol{\chi}}(X,\bk;q,a;\b{A})
:=\hspace{-0,5cm}
\sum_{\substack{(D_\bu) \in \c{D}(X,\bk;q,a)\\
\forall \bu  (A_\bu \leq D_\bu < \Delta A_\bu)\\
\forall \bu  (\omega(D_\bu)\leq \Omega)
}} 
\hspace{-0,3cm}
2^{-{\bk}\omega(D)}
\hspace{-0,1cm}
\l(\prod_{\bu,\bv \in (\mathbb{F}_2^2)^{\bk}}
\l(\frac{D_\bu}{D_\bv}\r)^{\!\Phi_{\bk}(\bu,\bv)}
\hspace{-0,1cm}
\r)
\hspace{-0,1cm}
\prod_{\bu \in (\mathbb{F}_2^2)^{\bk}}
\prod_{\chi \in T}\chi(D_\bu)^{\Psi_{\chi}(\bu)}
\]
and note that, in light of~\eqref{eq:sqrf},
we can deduce 
as in~\cite[Eq.(32)]{MR2276261}
that 
\beq{eq:32}
{
S_{\boldsymbol{\chi}}(X,\bk;q,a)
=
\sum_{\substack{\b{A}:\prod_\bu A_\bu \leq X}}
\hspace{-0,3cm}
S_{\boldsymbol{\chi}}(X,\bk;q,a;\b{A})
+
O(X(\log X)^{-1})
.}
The contribution 
towards~\eqref{eq:32}
of the first family, defined through
\beq{def:33}{
\prod_{\bu}A_\bu\geq \Delta^{-4^{\bk}} X
,}
can be proved
to be 
$\ll X(\log X)^{-1}$
with a similar argument
as the one leading to~\cite[Eq.(34)]{MR2276261}.
We now
let 
\[ 
X^\ddag:=
\min\big\{\Delta^\ell\geq  \exp\big((\log X)^{\epsilon 2^{-{\bk}}}\big)
\big\}
.\]
The contribution 
towards~\eqref{eq:32}
of those $\b{A}$ fulfilling that
\beq{def:37}{
\text{at most } 2^{\bk}-1 \text{ of the } A_\bu \text{ are larger than } X^\ddag
}
can be shown to be 
$
\ll X(\log X)^{\epsilon-2^{-{\bk}}}
$
as in~\cite[Eq.(39)]{MR2276261}. 

We next pass to arguments related to 
cancellation due to oscillation of characters,
in this case~\eqref{eq:trivbnd78} is not enough.
The exponents $\Phi_k(\bu,\bv)$
will now play a r\^{o}le. Following Fouvry and Kl\"{u}ners 
we call two indices $\bu,\bv$ \textit{linked}
if $\Phi_{\bk}(\bu,\bv)+\Phi_{\bk}(\bv,\bu)=1$.
We next 
define \[
X^\dag:=(\log X)^{3[1+4^{\bk}(1+2^{\bk})]}
\]
and consider the contribution of $\b{A}$ with
\beq{def:37}{
\prod_\bu A_\bu <\Delta^{-4^{\bk}}X
\text{ and for two linked } \bu \text{ and } \bv \text{ we have }  \min\{A_\bu,A_\bv\} \geq X^\dag
.}
Fouvry and Kl\"{u}ners treat this case
by drawing upon the important
work of Heath-Brown~\cite{MR1347489}
in the
form stated in~\cite[Lem.12]{MR2276261}. 
Specifically for $\b{A}$ as in~\eqref{def:37}
we have 
\[|S_{\boldsymbol{\chi}}(X,\bk;q,a;\b{A})|
\leq\sum_{(D_\bw)_{\bw\notin \{\bu,\bv\}}}
\Big(
\prod_{\bw\notin \{\bu,\bv\}}
2^{-{\bk}\omega(D_\bw)}
\Big)
\hspace{-0,9cm}
\sum_{\substack{
a_1,a_2 \in (\Z\cap(0,q])^2\\
\\
a_1a_2
\prod_{\bw \notin \{\bu,\bv\}} 
D_\bw
\equiv a\md{q}
}}
\hspace{-0,5cm}
\Big|
M((D_\bw))
\Big|
,\]
where
\[
M((D_\bw)):=\sum_{D_\bu,D_\bv}
\l(\frac{D_\bu}{D_\bv}\r)
g(D_\bu,(D_\bw)_{\bw\notin \{\bu,\bv\}})
g(D_\bv,(D_\bw)_{\bw\notin \{\bu,\bv\}})
,\]
\[
g(D_\bu,(D_\bw)_{\bw\notin \{\bu,\bv\}})
:=
\frac{\b{1}_{a_1,q}(D_\bu)}{2^{{\bk}\omega(D_\bu)}}
\hspace{-0,3cm}
\prod_{\bw\notin \{\bu,\bv\}}
\hspace{-0,2cm}
\l(\frac{D_\bu}{D_\bw}\r)^{\!\Phi_{\bk}(\bu,\bw)}
\hspace{-0,3cm}
\prod_{\bw\notin \{\bu,\bv\}}
\hspace{-0,2cm}
\l(\frac{D_\bw}{D_\bu}\r)^{\!\Phi_{\bk}(\bw,\bu)}
\hspace{-0,2cm}
\prod_{\chi \in T}\chi(D_\bu)^{\Psi_{\chi}(\bu)}
,\]
$\b{1}_{\alpha,\beta}$ denotes the indicator function of the set $\{m\in \Z:m\equiv \alpha \md{\beta}\}$
and similarly for 
$g(D_\bv,(D_\bw)_{\bw\notin \{\bu,\bv\}})$.
Since $|g(D_\bu,(D_\bw)_{\bw\notin \{\bu,\bv\}})
|,|g(D_\bv,(D_\bw)_{\bw\notin \{\bu,\bv\}})
|\leq 1$
the argument in~\cite[p.476]{MR2276261}
that validates~\cite[Eq.(42)]{MR2276261}
can be 
adopted
in the obvious way 
to yield
\[
\sum_{\b{A} \text{ fulfils }\eqref{def:37}}
|S_{\boldsymbol{\chi}}(X,\bk;q,a;\b{A})|
\ll
X(\log X)^{-1}
.\]
Note that we have used~\cite[Lem.15]{MR2276261}
for sequences satisfying $|a_m|,|b_n|\leq 1$ rather than
$|a_m|,|b_n|< 1$, however using~\cite[Lem.15]{MR2276261}
for $a_m/2,b_n/2$ in place of $a_m,b_n$ proves a
version of~\cite[Lem.15]{MR2276261}
under the more general 
assumption  $|a_m|,|b_n|<2$
and with the same conclusion.

The fourth family consists of $\b{A}$ fulfilling
$\prod_\bu A_\bu <\Delta^{-4^{\bk}}X$,
any linked $\bu,\bv$
satisfy the inequality
$\min\{A_\bu,A_\bv\} < X^\dag$
and there exist linked $\bu,\bv$ 
with $2\leq A_\bv \text { and } A_\bu\geq X^\ddag$.
Their contribution towards
$S_{\boldsymbol{\chi}}(X,\bk;q,a;\b{A})$
 is
\beq{def:4case}
{
\ll \max_{\substack{\sigma\md{q}\\ \gcd(\sigma,q)=1}}
\sum_{\substack{(D_\bw)_{\bw \notin \{\bu,\bv\}}\\ A_\bw\leq D_\bw< \Delta A_\bw}}
\sum_{\substack{D_\bv\\ A_\bv\leq D_\bv< \Delta A_\bv}}
|M_\sigma|
,}
where $M_\sigma$ is defined through
\[
\sum_{\substack{
D_\bu \equiv \sigma \md{q}
\\
A_\bu\leq D_\bu< \Delta A_\bu
}}
2^{-{\bk}\omega(D_\bu)}
\l(\frac{D_\bu}{D_\bv}\r) 
\prod_{\chi \in T}\chi(D_\bu)^{\Psi_{\chi}(\bu)}
\hspace{-0,1cm}
=
\hspace{-0,1cm}
\l(
\prod_{\chi \in T}\chi(\sigma)^{\Psi_{\chi}(\bu)}
\r)
\hspace{-0,2cm}
\sum_{\substack{
D_\bu \equiv \sigma \md{q}\\
A_\bu\leq D_\bu< \Delta A_\bu}}
2^{-{\bk} \omega(D_\bu)}
\l(\frac{D_\bu}{D_\bv}\r) 
.\]
Letting 
$P^+(m)$ denote the largest prime factor of a positive integer $m>1$
and setting $P^+(1):=1, 
m:=D_\bu/P^+(D_\bu)$
we obtain
\[
M_\sigma
\ll
\sum_{\substack{m P^+(m)<\Delta A_\bu\\(m,q)=1
}}
\frac{\mu(m)^2}{2^{{\bk} \omega(m)}}
\Big|
\sum_{\substack{
m p \equiv \sigma \md{q}
}}
\hspace{-0,5cm}
\mu\big(p m \prod_{\bw\neq \bu}D_\bw\big)^2
\l(\frac{p}{D_\bv}\r) 
\Big| 
,\]
where 
the inner sum is over primes $p$ 
with
$\max\{A_\bu/m,P^+(m)\}\leq p< \Delta A_\bu/m$.
We may now use Dirichlet characters to modulus $q$
to detect the congruence condition on $p$.
We will
subsequently
be faced with $\phi(q)$ new sums over $p$,
each one of which can be bounded via~\cite[Lem.13]{MR2276261}.
This furnishes
\[
\sum_{\substack{
m p \equiv \sigma \md{q}
}}
\hspace{-0,5cm}
\mu\big(p m \prod_{\bw\neq \bu}D_\bw\big)^2
\l(\frac{p}{D_\bv}\r) 
\ll 
\frac{A_\bv^{1/2}A_\bu}{m}
(\log X)^{-N \epsilon 2^{-{\bk}+1}}+\Omega,\]
valid for each large enough positive $N$ that is independent of $\b{A}$ and $m$.
The term $\Omega$ accounts for the presence of the $\mu^2$-terms.
Indeed, by~\eqref{eq:sqrf}
the number of distinct 
prime divisors of $m$ and each $D_\bw$ is at most $\Omega$. 
A moment's thought now
reveals that 
once the last bound is
injected
into~\eqref{def:4case}
and $N$ is suitably increased in comparison to $\bk$,
the contribution of $\b{A}$ in the fourth case is 
$
\ll
X (\log X)^{-1} 
$,
as in~\cite[Eq.(47)]{MR2276261}.

Let us now
introduce the conditions
\begin{equation} \lab{eq:48}
 \left\{  \begin{array}{ll}
                  &\prod_{\bu \in (\F_2^2)^k} A_\bu <\Delta^{-4^{\bk}}X,  \\
&\text{at least } 2^{\bk} \text{ indices satisfy } A_\bu>X^\ddag,\\
&\text{two indices } \bu \text{ and } \bv \text{ with } A_\bu,A_\bv>X^\dag \text{ are always linked,}\\
&\text{if } A_\bu  \text{ and } A_\bv   \text{ with } A_\bv\leq A_\bu \text{ are linked, then either } \\
&A_\bv=1  \text{ or } (2\leq A_\bv<X^\dag \text{ and } A_\bv\leq A_\bu < X^\ddag).
                \end{array}\right.\end{equation}  
Increasing the value of $A$ in comparison to ${\bk}$
and
assorting all estimates so far yields
\beq{prop:2}
{
S_{\boldsymbol{\chi}}(X,{\bk};q,a)
=
\sum_{\substack{\b{A} \text{ satisfies }\eqref{eq:48}}}
\hspace{-0,3cm}
S_{\boldsymbol{\chi}}(X,\bk;q,a;\b{A})
+
O(X(\log X)^{\epsilon-2^{-{\bk}}})
,}
which is in analogy with~\cite[Prop.2]{MR2276261}.

\subsection{The main term}
\label{s:themte}
We can now obtain the following as in~\cite[Prop.3]{MR2276261},
\beq{prop:2}
{
S_{\boldsymbol{\chi}}(X,{\bk};q,a)
=
\sum_{\substack{\b{A} \text{ satisfies }\eqref{eq:50}}}
\hspace{-0,3cm}
S_{\boldsymbol{\chi}}(X,\bk;q,a;\b{A})
+
O(X(\log X)^{\epsilon-2^{-{\bk}}})
,}
where
\begin{equation} \lab{eq:50}
 \left\{  \begin{array}{ll}
&\c{U}:=\{\bu:A_\bu>X^\ddag\}\text{ is a maximal subset of unlinked indices,} \\ 
&\prod_{\bu \in (\F_2^2)^{\bk}} A_\bu \leq \Delta^{-4^{\bk}}X \text{ and } A_\bu=1  \text{ for } \bu \notin \c{U}.
                \end{array}\right.\end{equation}  
Similarly to~\cite[Eq.(50)]{MR2276261}
we will say that 
$\b{A}$ is
\textit{admissible} for $\c{U}$ if 
$A_\bu>X^\ddag \Leftrightarrow \bu \in \c{U}$,
$A_\bu=1 \Leftrightarrow \bu \notin \c{U}$
and $\prod_{\bu \in (\F_2^2)^{\bk}} A_\bu \leq \Delta^{-4^{\bk}}X$.
Assume that $\b{A}$ is admissible for $\c{U}$ 
and note that $\#\c{U}=2^{\bk}$.
By quadratic reciprocity we obtain that 
$S_{\boldsymbol{\chi}}(X,\bk;q,a;\b{A})$ equals 
\begin{alignat*}{3} 
&\sum_{\substack{(h_\bu) \in (\Z/4\Z)^{2^{\bk}},  \prod_{\bu \in \c{U}}h_\bu\equiv 3 \md{4}}}       
   &&\l(\prod_{\bu,\bv \in \c{U}}(-1)^{\Phi_{\bk}(\bu,\bv)\frac{h_\bu-1}{2}\frac{h_\bv-1}{2}}\r) &\times  \\
& \sum_{\substack{(g_\bu) \in (\Z/q\Z)^{2^{\bk}}, \prod_{\bu \in \c{U}}g_\bu\equiv a \md{q}\\\forall \bu \in \c{U}(h_\bu \equiv g_\bu \md{4})}}
  &&\l(\prod_{\bu \in \c{U}  }\prod_{\chi \in T}\chi(g_\bu)^{\Psi_{\chi}(\bu)}\r)        &\times \\
&\!
  \sum_{\substack{(D_\bu) \in \N^{2^{\bk}}, \forall \bu \ (\omega(D_\bu)\leq \Omega)\\ 
\forall \bu \ (D_\bu \equiv g_\bu \md{q},A_\bu \leq D_\bu < \Delta A_\bu)}} 
  &&\l(\prod_{\bu \in \c{U}} 2^{-{\bk}\omega(D_\bu)}\r)\mu^2\!\l(\prod_{\bu \in \c{U}} D_\bu\r).        & 
\end{alignat*} 
We can evaluate the sum over $D_\bu$ via the estimate,
\begin{equation}
\label{eq:bavohaa}
\sum_{\substack{m\in \N \cap [y,Y]\\ \omega(m)=\ell \\ m\equiv g \md{q}}} \mu(n_0 m)^2
= \frac{1}{\phi(q)}
\sum_{\substack{m\in \N \cap [y,Y]\\ \omega(m)=\ell \\ \gcd(m,q)=1}} \mu(n_0 m)^2
+O_A\l( \frac{(\ell+1)^A}{Y^{-1}(\log 2Y)^{A}}+ \frac{\omega(n_0)}{Y^{-1+\frac{1}{\ell}}}\r),
\end{equation}
valid for each
square-free integer $n_0$
that is coprime to $q$,
$A>0,Y\geq y \geq 1, \ell \in \Z_{\geq 0}$,
where the implied constant depends at most on $A$. 
This can be proved 
in a similar way as~\cite[Lem.19]{MR2276261}
by replacing the congruence condition 
to modulus $4$
on $p_\ell$ in~\cite[Eq.(53)]{MR2276261}
by one to modulus $q$.
Applying~\eqref{eq:bavohaa} repeatedly
as in~\cite[p.g.481-482]{MR2276261}
to estimate
the sums over $D_\bu$ 
leads us to 
\begin{align*}
&\sum_{\substack{(D_\bu) \in \N^{2^{\bk}}, \forall \bu  (\omega(D_\bu)\leq \Omega)\\ 
\forall \bu  (D_\bu \equiv g_\bu \md{q},A_\bu \leq D_\bu < \Delta A_\bu)}} 
\hspace{-0,8cm}
\l(\prod_{\bu \in \c{U}} 2^{-{\bk}\omega(D_\bu)}\r)\mu^2\!\l(\prod_{\bu \in \c{U}} D_\bu\r)
\\=
\phi(q)^{-2^{\bk}}
\hspace{-0,6cm}
&\sum_{\substack{(D_\bu) \in \N^{2^{\bk}}, \forall \bu  (\omega(D_\bu)\leq \Omega)\\ 
\forall \bu  (A_\bu \leq D_\bu < \Delta A_\bu)}} 
\hspace{-0,0cm}
\l(\prod_{\bu \in \c{U}} 2^{-{\bk}\omega(D_\bu)}\r)\mu^2\!\l(q\prod_{\bu \in \c{U}} D_\bu\r)
+O(X(\log X)^{-1-4^{{\bk}}(1+2^{\bk})})
.\end{align*}  
Using this we obtain
as in~\cite[Eq.(55)]{MR2276261}
that for any
fixed admissible $\c{U}$ we have 
\begin{align*}
\sum_{\substack{\b{A} \text{ admissible for }\c{U}}}
\hspace{-0,3cm}
&S_{\boldsymbol{\chi}}(X,{\bk};q,a;\b{A})
=2^{-{\bk}}\phi(q)^{-2^{\bk}}
\hspace{-0,7cm}
\sum_{\substack{(h_\bu) \in (\Z/4\Z)^{2^{\bk}}\\ \prod_{\bu \in \c{U}} h_\bu\equiv 3 \md{4}}}   
\hspace{-0,4cm}
\l(\prod_{\bu,\bv \in \c{U}}(-1)^{\Phi_{\bk}(\bu,\bv)\frac{h_\bu-1}{2}\frac{h_\bv-1}{2}}\r)  \times \\
&
\hspace{-0,5cm}
\sum_{\substack{(g_\bu) \in (\Z/q\Z)^{2^{\bk}}, \prod_{\bu \in \c{U}}g_\bu\equiv a \md{q}\\\forall \bu \in \c{U}(h_\bu \equiv g_\bu \md{4})}}
   \l(\prod_{\bu \in \c{U}  }\prod_{\chi \in T}\chi(g_\bu)^{\Psi_{\chi}(\bu)}\r)              \times \\
&
\hspace{-0,5cm}
\ \ \ \ \
\sum_{\substack{(D_\bu) \in \N^{2^{\bk}}, \forall \bu \ (\omega(D_\bu)\leq \Omega)\\ 
\forall \bu \ (A_\bu \leq D_\bu < \Delta A_\bu) 
}} 
\hspace{-0,2cm}
\l(\prod_{\bu \in \c{U}} 2^{-{\bk}\omega(D_\bu)}\r)\mu^2\!\l(\mathrm{rad}(q)\prod_{\bu \in \c{U}} D_\bu\r)
+O\l(\frac{X}{\log X}\r),
\end{align*}
where the radical 
$\mathrm{rad}(m)$ stands for the product of the distinct prime divisors of an integer $m>1$.
We can now see that the condition $\omega(D_\bu)\leq \Omega$ can be ignored 
at the cost of an error term of size $\ll X(\log X)^{-1}$ as 
in the beginning of \S\ref{ss:ff}.
We can furthermore show 
as in~\cite[p.g.482]{MR2276261}
that 
\[\sum_{\substack{(D_\bu) \in \N^{2^{\bk}}\\ 
\forall \bu (A_\bu \leq D_\bu < \Delta A_\bu)}} 
\hspace{-0,4cm}
\l(\prod_{\bu \in \c{U}} 2^{-{\bk}\omega(D_\bu)}\r)
\hspace{-0,1cm}
\mu^2\!\l(\mathrm{rad}(q)\prod_{\bu \in \c{U}} D_\bu\r)
\hspace{-0,1cm}
=
\hspace{-0,1cm}
\sum_{m\leq X} \mu(\mathrm{rad}(q)m)^2+O
\hspace{-0,1cm}
\l(X (\log X)^{\epsilon-2^{-{\bk}}}\r)
.\]
It is easily proved via M\"{o}bius inversion
that for fixed $a,q>0$ with 
$\gcd(a,q)=1$
we have
\[
\sum_{m\leq X} \mu(\mathrm{rad}(q)m)^2
=
\frac{\phi(q)}{q} 
\Big(\prod_{p\nmid q}(1-p^{-2})\Big)
X
+O\l(\sqrt{X}\r)
\] 
and
\[ 
\sum_{\substack{m\leq X\\m\equiv a \md{q}}} \mu(m)^2
=
\frac{1}{q} 
\Big(\prod_{p\nmid q}(1-p^{-2})\Big)
X
+O\l(\sqrt{X}\r)
.\] 
Combining these yields
\[
\sum_{m\leq X} \mu(\mathrm{rad}(q)m)^2
=
\phi(q)
\sum_{\substack{m\leq X\\m\equiv a \md{q}}} \mu(m)^2
+O\l(\sqrt{X}\r)
.\] 
We thus obtain the following
for every maximal unlinked subset $\c{U}$, 
\[
\sum_{\substack{\b{A} \text{ admissible for }\c{U}}}
\hspace{-0,4cm}
S_{\boldsymbol{\chi}}(X,\bk;q,a;\b{A})
=
\frac{\gamma_{\boldsymbol{\psi}}(\c{U})}{2^{\bk}\phi(q)^{2^{\bk}-1}}
\Bigg(
\hspace{-0,1cm}
\sum_{\substack{m\leq X\\m\equiv a \md{q}}} \mu(m)^2
\Bigg)
+
O\hspace{-0,1cm}
\l(X (\log X)^{\epsilon-2^{-{\bk}}}\r)
,\]
where 
\[
\gamma_{\boldsymbol{\psi}}(\c{U})
:=
\hspace{-0,7cm}
\sum_{\substack{(h_\bu) \in (\Z/4\Z)^{2^{\bk}}  \\ \prod_{\bu \in \c{U}}h_\bu\equiv 3 \md{4}}}
\hspace{-0,4cm}
\l(\prod_{\bu,\bv \in \c{U}}(-1)^{\Phi_{\bk}(\bu,\bv)\frac{h_\bu-1}{2}\frac{h_\bv-1}{2}}\r)
\hspace{-0,4cm}
\sum_{\substack{(g_\bu) \in (\Z/q\Z)^{2^{\bk}}   \\ \prod_{\bu \in \c{U}}g_\bu\equiv a \md{q}
\\
\forall \bu \in \c{U}(h_\bu \equiv g_\bu \md{4})
}}
\hspace{-0,4cm}
\l(
\prod_{\bu\in \c{U}}
\prod_{\chi \in T}\chi(g_\bu)^{\Psi_{\chi}(\bu)}\r)
.\]
We can now infer via~\eqref{prop:2} 
that the last equation
proves 
\begin{equation}
\label{eq:there}
\frac{S_{\boldsymbol{\chi}}(X,\bk;q,a)}
{\#\big\{m \in [1,X]: q\mid m-a, \mu(m)^2=1\big\}}
= 
\l(
\sum_{\c{U}}\gamma_{\boldsymbol{\psi}}(\c{U})
\r)
\frac{\phi(q)^{1-2^{\bk}}}{2^{\bk}}
+
O((\log X)^{\epsilon-2^{-{\bk}}})
,\end{equation}
where $\c{U}$ ranges over
maximal
unlinked 
subsets of $\mathbb{F}_2^{2\bk}$.

\subsection{Simplifying $\b{\gamma_\psi(\c{U})}$}
\label{s:simplefcn}
Introduce the following
Dirichlet character $\md{n_1n_2}$,
\[\rho_\bu
:=
\prod_{\chi \in T}
\chi^{\Psi_{\chi}(\bu)}
.\] 

We will call a maximal set of unlinked indices $\c{U}$ 
\textit{stable}
if 
\[
\forall \chi \in T,
\forall \bu \in \c{U}
({\Psi_{\chi}(\bu)}=0)
\text{ or }
\forall \chi \in T,
\forall \bu \in \c{U}
({\Psi_{\chi}(\bu)}=1)
.\]
Let us now prove that  
$$ \sum_{\substack{(g_\bu) \in (\Z/q\Z)^{2^{\bk}}   \\ \prod_{\bu \in \c{U}}g_\bu\equiv a \md{q}
\\
\forall \bu \in \c{U}(h_\bu \equiv g_\bu \md{4})
}}
\hspace{-0,4cm}
\prod_{\bu\in \c{U}} \rho_\bu(g_\bu)
=
\mathbf{1}_{\mathcal{U}\text{ stable}}(\c{U})
\l(\frac{\phi(q)}{2}\r)^{2^{\bk}-1} 
.$$
Write $q=2^b n_0m$,
where
$b:=\nu_2(q)$,
$\gcd(n_0,n_1n_2)=1$
and $n_0$ has radical equal to $n_1n_2$. 
Define 
\[
U_1(n_0):=\{u \in \mathbb{Z}/n_0\mathbb{Z}:  u \equiv 1 \md{n_1n_2}\}
\text{ and }
U_1(2^{b}):=\{u \in \mathbb{Z}/2^{b}\mathbb{Z}:u \equiv 1 \md{4}\}.\]
Recalling the identification of groups
$(\mathbb{Z}/q\mathbb{Z})^{*}=U_1(2^b) \times (\mathbb{Z}/4\mathbb{Z})^{*} \times U_1(n_0) \times (\mathbb{Z}/n_1n_2\mathbb{Z})^{*}$,
we see that 
\[
\sum_{\substack{(g_\bu) \in (\Z/q\Z)^{2^{\bk}}   \\ \prod_{\bu \in \c{U}}g_\bu\equiv a \md{q}
\\
\forall \bu \in \c{U}(h_\bu \equiv g_\bu \md{4})
}}
\prod_{\bu\in \c{U}}\rho_\bu(g_\bu)
=
(\#U_1(2^b)\#U_1(n_0)\phi(m))^{2^{\bk}-1}
\hspace{-0,5cm}
\sum_{\substack{(m_\bu) \in (\Z/{n_1n_2}\Z)^{2^{\bk}}   \\ \prod_{\bu \in \c{U}}m_\bu\equiv a \md{n_1n_2}
}}
\prod_{\bu\in \c{U}}\rho_\bu(m_\bu)
.\] 
Note that we have $\prod_{\b{u} \in \c{U}}\rho_{\b{u}_{0}}(m_\b{u})=
\rho_{\bu_0}(a)=1$ owing to~\eqref{eq:donaldtrump}-\eqref{eq:donaldtrump3}. 
Therefore, fixing $\b{u}_0 \in \c{U}$, 
we have the following equality for any choice of $m_{\b{u}}$ in the above sum 
$$\prod_{\bu \in \c{U}}\rho_\bu(m_\bu)=\rho_{\b{u_0}}(\b{u}_0)\prod_{\bu \in \c{U}-\{\b{u}_0\}}\rho_{\b{u}}(\b{u})= \prod_{\bu \in \c{U}-\{\b{u}_0\}}
\Big(\frac{\rho_{\b{u}}(m_{\bu})}{\rho_{\b{u}_0}(m_{\b{u}})}\Big).
$$
Therefore 
$$\sum_{\substack{(m_\bu) \in (\Z/{n_1n_2}\Z)^{2^{\bk}}   \\ \prod_{\bu \in \c{U}}m_\bu\equiv a \md{n_1n_2}
}}
\prod_{\bu\in \c{U}}\rho_\bu(m_\bu)=\sum_{\substack{(m_\bu) \in ((\mathbb{Z}/{n_1n_2}\mathbb{Z})^{*})^{2^{\bk}-1}  
}}
\prod_{\bu\in \c{U}-\{\b{u}_0\}}\frac{\rho_{\b{u}}(m_{\bu})}{\rho_{\b{u}_0}(m_{\b{u}})}.
$$
But the last clearly splits as
$$\prod_{\bu\in \c{U}-\{\b{u}_0\}}\Big(\sum_{\substack{(m_\bu) \in (\mathbb{Z}/{n_1n_2}\mathbb{Z})^{*}
}}\frac{\rho_{\b{u}}(m_{\bu})}{\rho_{\b{u}_0}(m_{\b{u}})} \Big)=\prod_{\bu\in \c{U}-\{\b{u}_0\}}\Big(\sum_{\substack{(m_\bu) \in (\mathbb{Z}/{n_1n_2}\mathbb{Z})^{*}
}}\prod_{\chi \in T}\chi^{\psi_{\chi}(\b{u})-\psi_{\chi}(\b{u}_{0})}(m_{\bu}) \Big).
$$
Using that the set of $\chi$ in $T$ consists of a set of linearly independent characters, we obtain that each factor of the last product
vanishes if and only if $\psi_{\chi}$ is not constant on $\c{U}$, i.e. if and only if $\c{U}$ is not stable. In the stable case its value is $\phi(n_1n_2)^{2^{|\b{k}|_{1}-1}}$. Therefore we have proved that
\begin{align*}
\sum_{\substack{
(g_\bu) \in (\Z/q\Z)^{2^{\bk}}   
\\
\prod_{\bu \in \c{U}}g_\bu\equiv a \md{q}
\\
\forall \bu \in \c{U}(h_\bu \equiv g_\bu \md{4})
}}
\hspace{-0,3cm}
\prod_{\bu\in \c{U}}\rho_\bu(g_\bu)
&=
(\#U_1(2^b)\#U_1(n_0)\phi(m)\phi(n_1n_2))^{2^{\bk}-1}
\mathbf{1}_{\mathcal{U}\text{ stable}}(\c{U})\\
&=
\Big
(\frac{\phi(q)}{2}\Big)^{2^{\bk}-1}
\mathbf{1}_{\mathcal{U}\text{ stable}}(\c{U}),
\end{align*}
from which we deduce that  
\[
\sum_{\c{U}}
\gamma_{\boldsymbol{\psi}}(\c{U})
=
\l(\frac{\phi(q)}{2}\r)^{2^{\bk}-1}
\sum_{\substack
{
\c{U} \text{stable}
}}
\sum_{\substack{(h_\bu)_{\bu\in \c{U}} \in (\Z/4\Z)^{2^{\bk}}  \\ 
\prod_{\bu \in \c{U}}h_\bu\equiv 3 \md{4}
}}
\l(\prod_{\bu,\bv \in \c{U}}(-1)^{\Phi_{\bk}(\bu,\bv)\frac{h_\bu-1}{2}\frac{h_\bv-1}{2}}\r)
,\] 
where the pairs $\bu,\bv$ are unordered.
The inner sum is identical to the one appearing in the work of 
Fouvry and Kl\"{u}ners,
however the outer sum does not appear in their work. Define 
$$\gamma(\mathcal{U}):=\sum_{\substack{(h_\bu)_{\bu\in \c{U}} \in (\Z/4\Z)^{2^{\bk}}  \\ 
\prod_{\bu \in \c{U}}h_\bu\equiv 3 \md{4}
}}
\l(\prod_{\bu,\bv \in \c{U}}(-1)^{\Phi_{\bk}(\bu,\bv)\frac{h_\bu-1}{2}\frac{h_\bv-1}{2}}\r).
$$
We are left with proving
\beq{eq:leftwith}
{\sum_{\mathcal{U} \ \text{stable}}\gamma(\mathcal{U})=2^{2^{\bk}+\bk-1} \mathcal{N}_2(\bk-\#T)
}
and this will be our aim in \S\ref{s:combina}.
\subsection{Combinatorics} 
\label{s:combina} 
From ~\cite[Lem.18]{MR2276261} we know that the
maximal unlinked sets of indices $\mathcal{U}$ 
consist precisely of cosets of $\bk$-dimensional subspaces of $\mathbb{F}_2^{2\bk}$. 
Therefore stable $\mathcal{U}$ are 
cosets of $\bk$-dimensional subspace of $\mathbb{F}_2^{2\bk}$, where all the $\Psi_{\chi}$ vanish. 

Next, introduce the bilinear form on $\mathbb{F}_2^{2\bk}$ via
$$L(\b{u},\b{v}):=\sum_{j=0}^{\bk}\bu_{2j+1}(\bv_{2j+1}+\bv_{2j+2})
.$$
Using the the terminology from~\cite{MR2276261}, we say that a $\bk$-dimensional subspace, $\mathcal{U}_{0}$, of $\mathbb{F}_2^{2\bk}$ is \emph{good} if
$$L_{|_{\mathcal{U}_{0} \times \mathcal{U}_{0}}} \equiv \ 0.
$$
Recall that
the upshot of~\cite[Lem.22-25]{MR2276261}
is that 
$\gamma$ vanishes on all cosets of non-good subspaces, meanwhile the 
total
contribution from the set of cosets of a fixed
good subspace is $2^{2^{\bk}+\bk-1}$. This provides us with
$$\sum_{\mathcal{U} \ \text{stable}} \gamma(\mathcal{U})=2^{2^{\bk}+\bk-1}
\#\{\mathcal{U}_{0} \ \text{good}:\Psi_{\chi}(\mathcal{U}_0)=0 \ \text{for each} \ \chi \in T \}
.$$
Now, following the proof of~\cite[Lem.26]{MR2276261}, if $\{e_1,\cdots,e_{2\bk}\}$ denotes the standard basis of $\mathbb{F}_2^{2\bk}$, choose a new basis via
$$\{b_1,\cdots,b_{2\bk}\}=\{e_1+e_2,e_2,\cdots,e_{2j-1}+e_{2j},e_{2j},\cdots,e_{2\bk-1}+e_{2\bk},e_{2\bk}\}.$$
Then, with respect to the new basis, $L$ assumes the form
$$L(\b{x},\b{y})=\sum_{j=0}^{j-1}\b{x}_{2j+1}\b{y}_{2j+2}
.$$
In the proof of part (i) of ~\cite[Lem.25]{MR2276261} it is verified that, if $X$ consists of the subspace generated by 
$\{b_i:i \text{ odd}\}$ 
and $Y$ 
consists of the subspace generated by
$\{b_i:i \text{ even}\}$, 
the map sending 
$\mathcal{U}_{0} \mapsto \pi_{X}(\mathcal{U}_{0})$
where $\pi_X$ is the projection map 
$\mathbb{F}_2^{2\bk}=X \oplus Y \to X$
gives a bijection between good subspaces of $\mathbb{F}_2^{2\bk}$ and vector subspaces of $\mathbb{F}_2^{\bk}$. On the other hand,
 we are counting only good subspaces where
$\Psi_{\chi}$ vanishes for each $\chi \in T$. Observe that 
owing to~\eqref{eq:milkespresso} we have that 
 $\Psi_{\chi}$ are all constantly $0$ on $Y$,
hence they define $\#T$ linearly independent linear functions from $X$ to $\mathbb{F}_2$
which we will denote by the same letters. Therefore $\mathcal{U}_0 \to \pi_{X}(\mathcal{U}_0)$
provides
a bijection between good subspaces where all $\Psi_{\chi}$ vanish
and subspaces of $X$ where all  
$\Psi_{\chi}$ vanish. 
Given that
$\Psi_{\chi}:X\to \F_2$ are independent
we find that 
the
cardinality of such subspaces
is precisely $\mathcal{N}_2(\bk-\#T)$. 
This substantiates~\eqref{eq:leftwith},
which concludes the proof of Theorem~\ref{m.t. on multi-moments}.
\section{From the mixed moments to the distribution}\label{from m.m. to distr.}
This section is devoted to deduce Theorem~\ref{m.t. on distribution} from Theorem~\ref{m.t. on multi-moments}. We will follow an adaptation of a method used by Heath-Brown in \cite{MR1292115}. \\
As explained in \S\ref{section:special divisors}, 
Theorem ~\ref{m.t. on distribution} can be equivalently rephrased as a theorem about the distribution of the vector
$$D \mapsto (m_{\chi}(D))_{
\widehat{G}_{n_1} \times \widetilde{\widehat{G}}_{n_2}}.
$$
Namely consider for any positive integer $j$ and subspace $Y \subseteq \widehat{G}_{n_1} \times \widetilde{\widehat{G}}_{n_2}$, the vector 
$$\b{v}^{(j,Y)} \in \mathbb{Z}_{\geq 0}^{\widehat{G}_{n_1} \times \widetilde{\widehat{G}}_{n_2}}
,$$
defined as $\b{v}^{(j,Y)}_{\chi}=j$ if $\chi \in Y$ and $\b{v}^{(j,Y)}_{\chi}=j-1$ if $\chi \not \in Y$. Assign to $\b{v}^{(j,Y)}$ \emph{mass} 
$$\mu(\b{v}^{(j,Y)})=\mcl(A \in \mathcal{G}_2:\#A[2]=2^{j-1}) 
\frac{\# \epi(\mathbb{F}_2^{j-1},Y)}{\# \Hom(\mathbb{F}_2^{j-1},\widehat{G}_{n_1} \times \widetilde{\widehat{G}}_{n_2})}.
$$
On the other hand, assign to all other vectors $\b{v} \in \mathbb{Z}_{\geq 0}^{\widehat{G}_{n_1} \times \widetilde{\widehat{G}}_{n_2}}
$ mass equal to $0$. In 
Proposition~\ref{heuristic comp. of multimoments}
it
is shown that this equips
$ \mathbb{Z}_{\geq 0}^{\widehat{G}_{n_1} \times \widetilde{\widehat{G}}_{n_2}}$ 
with a probability
measure
satisfying the following \emph{moment equations}:
$$ \sum_{\b{v} \in \mathbb{Z}_{\geq 0}^{\widehat{G}_{n_1} \times \widetilde{\widehat{G}}_{n_2}} }2^{\b{v} \cdot \b{k}}\mu(\b{v})=C_{\b{k}},
$$
where 
for any $\b{k} \in  \mathbb{Z}_{\geq 0}^{\widehat{G}_{n_1} \times \widetilde{\widehat{G}}_{n_2}}$ 
we define 
$$C_{\b{k}}:=2^{\bk} 
\sum_{Y \subseteq \widehat{G}_{n_1} \times \widetilde{\widehat{G}}_{n_2} }\mathbb{P}_{(\b{k})}(Y)\mathcal{N}_{2}(\bk-\dim(Y))
$$
and where 
$\b{v} \cdot \b{k}$ denotes the inner product.

We begin the proof of
Theorem~\ref{m.t. on distribution}
by showing that 
the distribution $\mu$ is characterized by the moment equations
given above.
Indeed we show more, namely assume $x$ is a map $ \mathbb{Z}_{\geq 0}^{\widehat{G}_{n_1} \times \widetilde{\widehat{G}}_{n_2}} \to [0,1]$
satisfying for any $\b{k} \in \mathbb{Z}_{\geq 0}^{\widehat{G}_{n_1} \times \widetilde{\widehat{G}}_{n_2}}$ the moment relations
\beq{eq:letsfinishthis}
{\sum_{\b{v} \in \mathbb{Z}_{\geq 0}^{\widehat{G}_{n_1} \times \widetilde{\widehat{G}}_{n_2}}} 2^{\b{v} \cdot \b{k}}x(\b{v})=C_{\b{k}}.}
Observe that one has the trivial bound
$C_{\b{k}} \ll 2^{\bk}
\mathcal{N}_2(\bk)$,
which leads to
$C_{\b{k}} \ll 2^{\frac{\bk^2+4\bk}{4}}$.
Letting
$F(t):=\prod_{n=0}^{\infty}(1-t2^{-n})$,
we therefore see that 
for any 
$\b{k}  \in \mathbb{Z}_{\geq 0}^{\widehat{G}_{n_1} \times \widetilde{\widehat{G}}_{n_2}}$, 
the following
series
is absolutely convergent, 
\beq{eq:letsfinishthis2}
{
\sum_{\b{n} \in \mathbb{Z}_{\geq 0}^{\widehat{G}_{n_1} \times \widetilde{\widehat{G}}_{n_2}}} a_{\b{n}}
C_{\b{n}} 
2^{-\b{n}\cdot \b{k}}
,}
where $a_{\b{n}}$ is the $\b{n}$-coefficient of the Taylor expansion of 
$$\widetilde{F}(\b{z}):=\prod_{\chi \in \widehat{G}_{n_1} \times \widetilde{\widehat{G}}_{n_2}}F(z_{\chi})
.$$
Injecting~\eqref{eq:letsfinishthis} into~\eqref{eq:letsfinishthis2},
expanding in terms of $x$ 
and exchanging the order of summation, we obtain
$$\sum_{\b{n} \in \mathbb{Z}_{\geq 0}^{\widehat{G}_{n_1} \times \widetilde{\widehat{G}}_{n_2}}} a_{\b{n}}
C_{\b{n}} 
2^{-\b{n}\cdot \b{k}}=\sum_{\b{m} \in \mathbb{Z}_{\geq 0}^{\widehat{G}_{n_1} \times \widetilde{\widehat{G}}_{n_2}} }
\widetilde{F}((2^{{\b{m}_\chi}-{\b{k}_{\chi}}}))
x(\b{m}).
$$
If 
for all $\chi$ we have 
$\b{m}_\chi < \b{k}_\chi$ 
then $\widetilde{F}((2^{{\b{m}_\chi}-{\b{k}_{\chi}}}))\neq 0$,
otherwise we 
have 
$\widetilde{F}((2^{{\b{m}_\chi}-{\b{k}_{\chi}}}))=0$.
Therefore,
the right side is a finite sum supported in the region $ \b{m}_{\chi} < \b{k}_{\chi}$ for every 
$\chi$. 
Hence,
using the triangular system of relations above 
one can successively reconstruct the function
$x(\b{m})$ from the moments $C_{\b{k}}$. Therefore,
we necessarily have
$x(\b{m})=\mu(\b{m})$ 
described above.

Let 
$a,q$ be integers
as in
Theorem~\ref{m.t. on distribution}
and 
for any $\b{j} \in \mathbb{Z}_{\geq 0}^{\widehat{G}_{n_1} \times \widetilde{\widehat{G}}_{n_2}}$ and $X \in \mathbb{R}_{\geq 1}$, 
define the quantity
$d_{\b{j}}(X)$ as the proportion of  
all
positive
square-free
integers 
$D\leq X$
satisfying $D\equiv a\md{q}$
and
$\b{m}_\chi(D)=2^{\b{j}_\chi}$ for all $\chi$.
Therefore, 
Theorem~\ref{m.t. on multi-moments}
shows that 
for any $\b{k} \in \mathbb{Z}_{\geq 0}^{\widehat{G}_{n_1} \times \widetilde{\widehat{G}}_{n_2}}$ we have
$ \sum_{\b{r}}d_{\b{r}}(X)2^{\b{r}\cdot \b{k}}=C_{\b{k}}
+o(1),
\text{ as } X\to+\infty
$,
where the sum is taken
over $\b{r} \in \mathbb{Z}_{\geq 0}^{\widehat{G}_{n_1} \times \widetilde{\widehat{G}}_{n_2}}$. 
The argument concludes as follows:  fix any vector $\b{v} \in \mathbb{Z}_{\geq 0}^{\widehat{G}_{n_1} \times \widetilde{\widehat{G}}_{n_2}}$; by compactness of the interval $[0,1]$ and a standard diagonal argument, 
one can choose a sequence $\{Y_n\}_{n \in \N}$ tending to infinity, such that $d_{\b{v}}(Y_n)$ converges to any of the limit points of 
$\{d_{\b{v}}(X):X\in \mathbb{R}_{\geq 1}\}$, 
call it $d'_{\b{v}}$, while 
for every other $\b{w}$ the sequence
$d_{\b{w}}(Y_n)$ is also converging to some limit point $d'_{\b{w}}$. 
Next, we fix $\b{h} \in \mathbb{Z}_{\geq 0}^{\widehat{G}_{n_1} \times \widetilde{\widehat{G}}_{n_2}}$, and we use the
previous
moment relation 
for $\b{k}=2\b{h}$, trivially bounding each terms with the total sum,
providing
$ d_{\b{r}}(Y_n) \ll_{\b{h}} 2^{-\b{r} \cdot \b{h}}$.
This enables us to
apply the dominated convergence theorem to
exchange the sum and the limit
in the expression of the $\b{h}$-th moment, from which
we deduce that $d'_{\b{w}}$ satisfies the following
moment equations as well:
$$\sum_{\b{w} \in \mathbb{Z}_{\geq 0}^{\widehat{G}_{n_1} \times \widetilde{\widehat{G}}_{n_2}} }2^{\b{w} \cdot \b{h}}d'_{\b{w}}=C_{\b{h}}.
$$  
We must therefore have 
$d'_{\b{w}}=\mu(\b{w})$
for all $\b{w} \in \mathbb{Z}_{\geq 0}^{\widehat{G}_{n_1} \times \widetilde{\widehat{G}}_{n_2}}$. 
Note that
$d'_{\b{v}}$ was an arbitrary limit point of $d_{\b{v}}(X)$, hence
we deduce that
$$ \lim_{X \to \infty} d_{\b{v}}(X)=\mu(\b{v}).
$$
Since $\b{v}$ was chosen arbitrarily in $\mathbb{Z}_{\geq 0}^{\widehat{G}_{n_1} \times \widetilde{\widehat{G}}_{n_2}}$ we have 
thus
shown that 
Theorem ~\ref{m.t. on distribution} holds, thereby
concluding the proof of Theorem~\ref{m.t. on distribution}.

\end{document}